\documentclass[12pt,english,a4paper]{amsart}

\input xy
\xyoption{all}

\usepackage{fullpage}
\usepackage{amssymb}
\usepackage{hyperref}
\usepackage{color} 

\newcommand{\Cc}{\mathbb{C}} 
\newcommand{\Hh}{\mathbb{H}} 
\newcommand{\Pp}{\mathbb{P}}
\newcommand{\Rr}{\mathbb{R}}
\newcommand{\Nn}{\mathbb{N}}
\newcommand{\Zz}{\mathbb{Z}}
\newcommand{\Qq}{\mathbb{Q}} 
\newcommand{\Ff}{\mathbb{F}} 

{\theoremstyle{plain}
\newtheorem{theorem}{Theorem}[section]   
\newtheorem{lemma}[theorem]{Lemma}       
\newtheorem{corollary}[theorem]{Corollary}   
\newtheorem{condition}[theorem]{Condition}  
         
}
{\theoremstyle{remark}
\newtheorem{definition}[theorem]{Definition}      
\newtheorem{remark}[theorem]{Remark}   
}

\def\Gal{\hbox{\rm Gal}}
\def\lhfl#1#2{\smash{\mathop{\hbox to 12mm{\leftarrowfill}}
\limits^{#1}_{#2}}}
\def\rhfl#1#2{\smash{\mathop{\hbox to 12mm{\rightarrowfill}}
\limits^{#1}_{#2}}}
\def\G-{regularly }

\DeclareMathOperator\PGL{{\rm{PGL}}}

\newcommand{\mP}{\mathbb{P}}
\newcommand{\ra}{\rightarrow}
\newcommand\divides{\mid}
\newcommand\sH{{\sf H}}
\newcommand\bC{{\bf C}}
\newcommand{\mZ}{\mathbb{Z}} 
\newcommand{\nn}{{\mathbb{N}}}  
\newcommand{\qq}{{\mathbb{Q}}}   
\newcommand{\zz}{{\mathbb{Z}}}
\newcommand{\HGlerok}{{\sf H}_{G,\le r_0}(k)} 
\newcommand{\HGrCk}{{\sf H}_{G,r}({\bf C})(k)} 
\newcommand{\HGRCk}{{\sf H}_{G,R}({\bf C})(k)}
\newcommand{\HGlero}{{\sf H}_{G,\le r_0}} 
\newcommand{\HGrC}{{\sf H}_{G,r}({\bf C})} 
\newcommand{\HGRC}{{\sf H}_{G,R}({\bf C})}

\begin{document}

\title[Rational pullbacks of Galois covers]{Rational pullbacks of Galois covers}

\author{Pierre D\`ebes}
\email{pierre.debes@univ-lille.fr}

\author{Joachim K\"onig}
\email{jkoenig@knue.ac.kr}

\author{Fran\c cois Legrand}
\email{francois.legrand@tu-dresden.de}

\author{Danny Neftin}
\email{dneftin@technion.ac.il}

\address{Laboratoire de Math\'ematiques Paul Painlev\'e, Universit\'e de Lille, 59655 Villeneuve d'Ascq Cedex, France}

\address{Department of Mathematics Education, Korea National University of Education, 28173 Cheongju, South Korea}

\address{Institut f\"ur Algebra, Fachrichtung Mathematik, TU Dresden, 01062 Dresden, Germany}

\address{Department of Mathematics, Technion, Israel Institute of Technology, Haifa 32000, Israel}



\thanks{{\it Acknowledgment}. This work was supported in part by the Labex CEMPI  (ANR-11-LABX-0007-01) and ISF grants No. 577/15 and No. 696/13. }

\begin{abstract} 
The finite subgroups of ${\rm PGL}_2(\Cc)$ are shown to be the only finite groups $G$ with this property: for some integer $r_0$ (depending on $G$), all Galois covers $X\rightarrow \Pp^1_\Cc$ of group $G$ can be obtained by pulling back those with at most $r_0$ branch points along \hbox{non-constant} rational maps $\Pp^1_\Cc \rightarrow \Pp^1_\Cc$. For $G\subset {\rm PGL}_2(\Cc)$, it is in fact enough to pull back one well-chosen cover with at most $3$ branch points. A consequence of the converse for inverse Galois theory  is that, for $G\not \subset {\rm PGL}_2(\Cc)$, letting the branch point number grow provides truly new Galois realizations $F/\Cc(T)$ of $G$. Another application is that the ``Beckmann--Black'' property that ``any two Galois covers of $\Pp^1_\Cc$ with the same group $G$ are always pullbacks of another Galois cover of group $G$'' only holds if $G\subset {\rm PGL}_2(\Cc)$.
\end{abstract}

\maketitle

\vspace{-2mm}

\section{Introduction} \label{sec:intro}
Suppose $f:X\rightarrow \Pp^1_k$ is a {\it $k$-regular cover}, i.e., a  
(branched) cover over a field $k$ with $X$ a smooth and geometrically
irreducible curve over $k$.
By a {\it rational pullback} of $f$,
we mean a $k$-regular cover $f_{T_0}:X_{T_0}\rightarrow \Pp^1_k$ obtained by pulling back $f$ along some
non-constant rational map $T_0:\Pp^1_k\rightarrow \Pp^1_k$.
If $f$ is given by a polynomial \hbox{equation $P(t,y)=0$} (with $f$ corresponding to the $t$-coordinate projection) and $T_0$ is viewed as a rational function $T_0(U)\in k(U)$, then an equation for the pullback $f_{T_0}$ is merely $P(T_0(u),y)=0$. 
As recalled in \S \ref{ssec:basics_2}, if $f$ is additionally Galois of group $G$, then for ``many'' $T_0$, the resulting map $f_{T_0}$ remains a $k$-regular Galois cover of group $G$.

The {\it Regular Inverse Galois Problem} over $k$ precisely consists in realizing each finite group as the Galois group of a $k$-regular Galois cover of $\Pp^1_k$. Rational pullback creates such covers, if one is already known. Finding  covers $g$ that are {\it not} rational pullbacks of some Galois covers $f$ of given group $G$ may be a more important issue. Indeed, if none of the $f$ are defined over $k$ (as a regular cover), the pullbacks $f_{T_0}$ are generally not expected to be either, in which case the remaining hope to realize $G$ as a regular Galois group over $k$ rests on those covers $g$.

\subsection{Main results} \label{ssec:intro_1}

Assume first $k=\Cc$. Consider the situation that all Galois covers $X\rightarrow \Pp^1_{\Cc}$ of given group $G$ can be obtained from a proper subset of them by rational pullback; say then that the subset is {\it regularly parametric}\footnote{The term ``regularly'' will be fully justified with the general definition of ``$k$-regular parametricity'' for which the base field $k$ is not necessarily algebraically closed (see Definition \ref{def:main}).}. For some finite groups $G$, a single cover $f$ may suffice. For example, the degree $2$ cover $\Pp^1_{\Cc}\rightarrow \Pp^1_{\Cc}$ sending $z$ to $z^2$ {\it is} regularly parametric. Such situations are, however, exceptional. For ``general'' finite groups, an opposite conclusion holds:
\begin{theorem} \label{thm:main1}
The finite subgroups of ${\rm PGL}_2(\Cc)$ (i.e., cyclic and dihedral groups, $A_4$, $S_4$, and $A_5$) are exactly those finite groups which have a regularly parametric cover $X\rightarrow \Pp^1_\Cc$. More precisely, given a finite group $G$, the following two statements hold:

\vskip 0,5mm

\noindent
{\rm (a)} if $G \subset {\rm PGL}_2(\Cc)$ \footnote{Above and throughout the paper, the condition ``$G \subset {\rm PGL}_2(\Cc)$'' (resp., ``$G \not\subset {\rm PGL}_2(\Cc)$'') really means that $G$ is {\it isomorphic} (resp., is {\it not isomorphic}) to a subgroup of ${\rm PGL}_2(\Cc)$.}, then $G$ has a regularly parametric cover,

\vskip 0,5mm

\noindent
{\rm (b)} if $G\not\subset {\rm PGL}_2(\Cc)$, then even the set of all Galois covers $X\rightarrow \Pp^1_\Cc$ of group $G$ and with at 
most $r_0$ branch points  is not regularly parametric, for any $r_0\geq 0$.
\end{theorem}

Hence, for $G\not\subset {\rm PGL}_2(\Cc)$, letting the branch point number grow provides an endless source of ``new'' Galois covers of $\mathbb{P}^1_\Cc$ of group $G$, {\rm i.e.}, not mere rational pullbacks of covers with a bounded branch point number, and so truly new candidates to be defined over $\Qq$.

Both statements of Theorem \ref{thm:main1} are non-trivial. The one showing that finite subgroups of ${\rm PGL}_2(\Cc)$ have a regularly parametric cover over $\Cc$ (with at most $3$ branch points) was proved in \cite[Corollary 2.5]{Deb18}, as a consequence of the {\it{twisting lemma}} and Tsen's theorem. The statement about  ``general'' finite groups, those not contained in ${\rm PGL}_2(\Cc)$, is a new result of this paper. In particular, it solves \cite[Problem 2.14]{Deb18}. 

Here is a more precise version. Given an integer $r\geq 0$ and an $r$-tuple ${\bf C}$ of non-trivial conjugacy classes of $G$, denote the stack of regular Galois covers $X \rightarrow \mathbb{P}^1$ of group $G$ with $r$ branch points by ${\sf H}_{G,r}$, and the stack of those with ramification type $(r,{\bf C})$ by ${\sf H}_{G,r}({\bf C})$ (see \S \ref{ssec:basics_1});
these are the  {\it Hurwitz stacks}. From Theorem \ref{thm:main1}, if $G\not\subset {\rm PGL}_2(\Cc)$, the set
$${{\sf H}_{G,\leq r_0}(\Cc) = \bigcup_{r\leq r_0} {\sf H}_{G,r}(\Cc)}$$
is never regularly parametric ($r_0\geq 1$). More precisely, we have the following:

\begin{theorem} \label{thm:main1-plus}
Let $k$ be an algebraically closed field of characteristic $0$ and let $G$ be a finite group, not contained in ${\rm PGL}_2(\Cc)$. Fix an integer $r_0\geq 0$. For every suitably large integer $R$, {depending on $r_0$}, there is a non-empty Hurwitz stack ${\sf H}_{G,R}({\bf C})$ such that {\bf not all} $k$-covers in ${\sf H}_{G,R}({\bf C})$ are rational pullbacks of $k$-covers in ${\sf H}_{G,\leq r_0}$.
\end{theorem}

Finite subgroups of ${\rm PGL}_2(\Cc)$, which are excluded in Theorem \ref{thm:main1-plus}, are the Galois groups of genus $0$ Galois covers. We show further that, if we also exclude the Galois groups of genus $1$ Galois covers, then the conclusion of Theorem \ref{thm:main1-plus} holds for {\textbf{all}} Hurwitz stacks ${\sf H}_{G,R}({\bf C})$ with {\bf C} of suitably large length $R$ (see Theorem \ref{lem:gge2}). In fact, this stronger conclusion can even be obtained without the extra genus $1$ assumption as long as the field $k$ is also assumed to be uncountable (see Remark \ref{rem:gen1_overC}).

A related result (see Theorem \ref{thm:new}) provides the following more explicit conclusion under the stronger assumption that the finite group $G$ has at least $5$ maximal non-conjugate cyclic subgroups (finite subgroups of ${\rm PGL}_2(\Cc)$ have at most $3$): {\it given any integer $r_0\geq 0$, for every suitably large even in\-te\-ger $R$, there is a non-empty Hurwitz stack  ${\sf H}_{G,R}({\bf C})$ such that {\textbf{no}} $k$-cover in ${\sf H}_{G,R}({\bf C})$ is a rational pullback of some $k$-cover in ${\sf H}_{G,\leq r_0}$}. 

If, instead of the stacks ${\sf H}_{G,r}$, we consider the smaller stacks ${\sf H}_{G,r}({\bf C})$, we obtain the following striking conclusion.
\begin{theorem}\label{introthm:r+1} 
Let $k$ be an algebraically closed field of characteristic $0$, let $G$ be a finite group not contained in $\PGL_2(\Cc)$, and let $(R,\bf C)$ be a ramification type for $G$ with $R\geq 4$. Then there exists a ramification type $(R+1,\bf D)$ for $G$  such that $\sH_{G,R+1}(\bf D)\not=\emptyset$ and no $k$-cover in $\sH_{G,R+1}(\bf D)$ is a pullback of a $k$-cover in $\sH_{G,R}(\bf C)$. 
\end{theorem} 

\begin{remark} \label{rk:intro}
(a) The assumption that $k$ is algebraically closed can be weakened in some of the results above
to only assume that $k$ is {\rm ample}, or even arbitrary (of characteristic $0$) in some situations; see Theorem \ref{thm:main-general-fields}. Recall that a field $k$ is {\it ample} if every geometrically irreducible smooth $k$-curve has either zero or infinitely many $k$-rational points. Ample fields include separably closed fields, Henselian fields, fields $\Qq^{{\rm tot} \Rr}$, $\Qq^{{\rm tot} p}$ of totally real or $p$-adic algebraic numbers. See, e.g., \cite{Jar11, BSF13, Pop14} for more on ample fields.

\vskip 1mm

\noindent
(b) Theorem \ref{introthm:r+1} also holds for $R=3$. The proof uses the same tools and techniques as for the case $R\geq 4$, but is longer and more technical; it is given in \cite{DKLN18}.  
\end{remark}

\subsection{Application} \label{ssec:intro_2}
Theorem \ref{thm:main1-plus} has the following consequence. Given an algebraically closed field $k$ of characteristic $0$, denote the set of all Galois covers $X\rightarrow \Pp^1_k$ of group $G$ by ${\sf H}_{G}(k)$. Say that a finite group $G$ has the {\it Beckmann--Black regular lifting property} over $k$ if, for any two  $g_1$ and $g_2$ in ${\sf H}_{G}(k)$, there exist $f \in {\sf H}_{G}(k)$ and two non-constant rational maps $T_{01}, T_{02}:\Pp^1_k\rightarrow \Pp^1_k$ such that $g_i=f_{T_{0i}}$ ($i=1,2$).

\begin{corollary} \label{cor:intro}
Let $k$ be an algebraically closed field of characteristic $0$. The finite subgroups of ${\rm PGL}_2(\Cc)$ are exactly those finite groups for which the {\it Beckmann--Black regular lifting property over $k$} holds.
\end{corollary}

The proof combines Theorem \ref{thm:main1-plus} with \cite[Theorem 2.1]{Deb18}, and is given in \S\ref{sssec:bounds} where the latter is recalled.

Our lifting property is a geometric variant of the Beckmann--Black {\it arithmetic} lifting property for which 
the cover $f\in {\sf H}_{G}(k)$ is requested to be defined over $\Qq$ and to {\it specialize} to some given Galois extensions $E_1/\Qq, \ldots, E_N/\Qq$ of group $G$ at some points $t_{01},\ldots, t_{0N}\in \Pp^1(\Qq)$  \footnote{The case $N=1$ is particularly significant as it supports Hilbert's strategy to solve the Inverse Galois Problem by first producing a $\mathbb Q$-regular Galois cover $f:X\rightarrow \Pp^1_\Qq$ of given group. It is known to hold for some groups: abelian, $S_n$, $A_n$, dihedral of order $2n$ with $n>1$ odd, etc.}.
There is no known counter-example to the latter. 
In comparison, our geometric variant is obvious for $N=1$ (by picking $T_0(U)=U$ so that $f_{T_0}=f$), and Corollary \ref{cor:intro} shows that it fails for $N\geq 2$ if $G\not\subset {\rm PGL}_2(\Cc)$. 

An intermediate stage towards counter-examples to the Beckmann--Black arithmetic property is to find groups $G$
{\it with no $\Qq$-parametric covers}, \hbox{i.e.} such that {no cover $f \in {\sf H}_{G}(\Qq)$ specializes to all Galois extensions of $\Qq$ of group $G$}. First examples were given in \cite{KL18, KLN19}. 
Extending them to all finite subgroups $G\not\subset {\rm PGL}_2(\Cc)$ is a next challenge, to which we will devote a
subsequent work. Theorem \ref{introthm:r+1} will be a key ingredient, the central idea being to combine the 
strong non regular parametricity conclusions of Theorem \ref{introthm:r+1} with a strategy  from \cite{Deb18} designed to deduce non $\Qq$-parametricity conclusions.

\subsection{Methods and organization of the paper}
Riemann's existence theorem (RET) is the fundamental theorem the above results build on. 
Furthermore, our proof of Theorem \ref{thm:main1-plus} exploits the geometric structure of the Hurwitz moduli spaces $\mathcal H_{G,R}(\bf C)$, namely, it shows that the dimension of the subset of all covers obtained by pulling back a cover in $\sH_{G,\leq r_0}$ is strictly smaller than that of $\mathcal H_{G,R}({\bf C})$, for sufficiently large  $R$. 

The tools used to bound the dimension of the subset of pullbacks include bounding the degree of defining polynomials for covers using a Riemann--Roch based result of Sadi \cite{Sad99}, Chevalley's theorem, and combinatorial ramification arguments. For the a\-na\-lo\-gous result over ample fields $k$, we show that $\sH_{G,R}({\bf C})(k)$ is non-empty, using Pop's $\frac{1}{2}$-Riemann existence theorem \cite{Pop94}, 
thus Zariski-dense in at least one connected component of $\mathcal H_{G,R}({\bf C})$; and hence
 this set is not contained in the above smaller dimension subset of pullbacks. Over arbitrary fields of characteristic $0$, we extend the theorem to certain families of groups using the rigidity method.

On the other hand, the proofs of Theorems \ref{introthm:r+1} and \ref{thm:new} construct explicit ramification types  whose Hurwitz stacks are non-empty by RET, and contain none of the pullbacks in question, as shown using combinatorial arguments,  the Riemann--Hurwitz formula, and Abhyankar's lemma. For an analogous result to Theorem \ref{thm:new} over ample fields $k$, we use  $\frac{1}{2}$-RET once more to ensure that  $\sH_{G,R}({\bf C})(k)\neq \emptyset$  for the constructed ramification types.

See \S\ref{sec:RP} for the proof of Theorem \ref{thm:main1-plus} and its variants. 
See  \S\ref{sec:RP2} for the proof of Theorem \ref{introthm:r+1}. 
\S \ref{sec:notation} is a preliminary section providing the basic notation and terminology together with some general prerequisites.

\section{Notation, terminology, and prerequisites} \label{sec:notation}

\subsection{Basic terminology {\rm (for more details, see \cite{DD97a} and \cite{DL13})}} \label{ssec:basics_1}

The base field $k$ is always assumed to be of characteristic $0$. We also fix a big algebraically closed field containing $k$ and the indeterminates that will be used, and in which every field compositum should be understood. 

\subsubsection{Covers} \label{sssec:covers}

Given a field $k$, a \textit{\textbf {$k$-variety}} is a geometrically irreducible and geometrically reduced quasiprojective $k$-scheme. A \textit{\textbf {$k$-curve}} is a $k$-variety of dimension $1$.

A field extension $F/k(T)$ is \textit{\textbf {$k$-regular}} if $F\cap \overline k= k$. A {\textit{{\textbf{$k$-regular cover}}}} $f:X\rightarrow \Pp^1_k$ is a non-constant finite morphism with $X$ a smooth $k$-curve; the function field extension $k(X)/k(T)$ is then $k$-regular. If in addition $k(X)/k(T)$ is Galois, then
$f:X\rightarrow \Pp^1_k$ is called a \textit{\textbf {$k$-regular Galois cover}}. If $k$ is algebraically closed,
we sometimes omit the word ``$k$-regular''.

We also use \textit{\textbf {affine equations}}: we mean the irreducible polynomial $P\in k[T,Y]$ of a primitive element of $k(X)/k(T)$, integral over $k[T]$. We say \textit{\textbf{defining equation}} if the primitive element is not necessarily integral over $k[T]$; then $P\in k(T)[Y]$. 

By \textit{\textbf {group}} and \textit{\textbf {branch point set}} of a $k$-regular cover $f:X\rightarrow \Pp^1_k$, we mean those of the extension $\overline k(X)/\overline k(T)$ \footnote{which is the function field extension associated with $f \otimes_k \overline{k} : X \otimes_k \overline{k} \rightarrow \mathbb{P}^1_{\overline{k}}$.}: the {\it{group}} of $\overline k(X)/\overline k(T)$ is the Galois group of its Galois closure and the {\it{branch point set}} of $\overline k(X)/\overline k(T)$ is the (finite) set of points $t\in \Pp^1(\overline k)$ such that the associated discrete valuations are ramified in $\overline{k}(X)/\overline k(T)$.

The field $k$ being of characteristic $0$, we also have the \textit{\textbf {inertia canonical invariant}} ${\bf C}$ of the $k$-regular cover $f:X\rightarrow \Pp^1_k$, defined as follows. If ${\bf t}=\{t_1,\ldots,t_r\}$ is the branch point set of $f$, then ${\bf C}$ is an $r$-tuple $(C_1,\ldots,C_r)$ of conjugacy classes of the group $G$ of $\overline k(X)/\overline k(T)$: for $i=1,\ldots,r$, the class $C_i$ is the conjugacy class of the distinguished\footnote{in the sense that they correspond to $e^{2i\pi/e_i}$ in the canonical isomorphism $I_{\mathfrak P} \rightarrow \mu_{e_i} =\langle e^{2i\pi/e_i} \rangle$.} generators of the inertia groups $I_{\mathfrak P}$ above $t_i$ in the Galois closure of $\overline k(X)/\overline k(T)$. The pair $(r,{\bf C})$ is called the \textit{\textbf {ramification type}} of $f$. More generally, given a finite group $G$, we say that a pair $(r,{\bf C})$ is \textit{a ramification type for $G$ over $k$} if it is the ramification type of at least one $k$-regular Galois cover $f:X\rightarrow \Pp^1_k$ of group $G$.

We also use the notation ${\bf e}=(e_1,\ldots,e_r)$ for the $r$-tuple with $i$th entry the ramification index $e_i=|I_{\mathfrak P}|$ of primes above $t_i$; $e_i$ is also the  order of elements of $C_i$, $i=1,\ldots,r$.

We say that two $k$-regular covers $f:X\rightarrow \Pp^1_k$ and $g:Y\rightarrow \Pp^1_k$ are \textit{\textbf{$\Pp^1_k$-isomorphic}} if there is an isomorphism $\chi: X\rightarrow Y$ defined over $k$ such  that $f= g\circ \chi$.

\subsubsection{Hurwitz stacks and Hurwitz spaces} \label{sssec:Hurwitz}

Given a finite group $G$, an integer $r\geq 1$, an $r$-tuple ${\bf C}$ of non-trivial conjugacy classes of $G$, and a field $k$ (of characteristic zero), we use the following notation:

\vskip 0.5mm

\noindent
- ${\sf H}_{G}(k)$: set of all $k$-regular Galois covers $f:X\rightarrow \Pp^1_k$ of group $G$,

\vskip 0.5mm

\noindent
- ${\sf H}_{G,r}(k)$ (resp., ${\sf H}_{G,\leq r}(k)$): subset of ${\sf H}_{G}(k)$ defined by the extra condition that the branch point number is $r$ (resp., that the branch point number is $\leq r$),

\vskip 0.5mm

\noindent
- ${\sf H}_{G,r}({\bf C})(k)$: subset of ${\sf H}_{G,r}(k)$ defined by the extra condition that the inertia canonical invariant is ${\bf C}$.

\vskip 0,5mm

The sets ${\sf H}_{G,r}(k)$ and ${\sf H}_{G,r}({\bf C})(k)$ can be viewed as the sets of $k$-rational points on some {\it stacks} ${\sf H}_{G,r}$ and ${\sf H}_{G,r}({\bf C})$, usually called \textit{\textbf {Hurwitz stacks}}. More formally, ${\sf H}_{G,r}(k)$ is the category whose objects are the $k$-regular Galois covers $f:X\rightarrow \Pp^1_k$ with $r$ branch points and given with an isomorphism $G\rightarrow {\rm Gal}(k(X)/k(T))$, 
and morphisms are the $\Pp^1_k$-isomorphisms commuting with the action of $G$; and similarly for ${\sf H}_{G,r}({\bf C})$.

We use the phrase {\it sets of $k$-points on the Hurwitz stacks ${\sf H}_{G,r}$ and  ${\sf H}_{G,r}({\bf C})$}  for the sets ${\sf H}_{G,r}(k)$ and ${\sf H}_{G,r}({\bf C})(k)$, but shall not use the stack structure.

On the other hand, we shall need the structure of variety of the associated moduli spaces, notably in \S\ref{ssec:Proof-1a} for the proof of Theorem \ref{thm:main1-plus}. The stacks ${\sf H}_{G,r}$ and ${\sf H}_{G,r}({\bf C})$ have indeed a coarse moduli space, which we denote by $\mathcal{H}_{G,r}$ and $\mathcal{H}_{G,r}(\bf C)$, respectively. They are commonly referred to as \textit{\textbf {Hurwitz spaces}}; see \cite{FV91} for more details\footnote{Due to our definition of the categories ${\sf H}_{G,r}(k)$ and  ${\sf H}_{G,r}({\bf C})(k)$, it is the so-called {\it inner} version of Hurwitz spaces that we shall be working with.}. They are finite unions of $r$-dimensional varieties and have this property: if $k$ is an algebraically closed field of characteristic $0$, the sets $\mathcal{H}_{G,r}(k)$ and $\mathcal{H}_{G,r}({\bf C})(k)$ are in one-one correspondence with the sets of isomorphism classes of objects in  the categories ${\sf H}_{G,r}(k)$ and ${\sf H}_{G,r}({\bf C})(k)$, respectively. Finally, if $k$ is a not necessarily algebraically closed field (but is still of characteristic $0$) and $f\in {\sf{H}}_{G,r}(k)$, then its isomorphism class $[f]$ still corresponds to a $k$-rational point of $\mathcal{H}_{G,r}({\bf C})$. If additionally $G$ has trivial center, then the spaces $\mathcal{H}_{G,r}(\bf C)$ are in fact fine moduli spaces, whence conversely any $k$-rational point on $\mathcal{H}_{G,r}({\bf C})$ corresponds to a $k$-regular Galois cover.
 
\subsection{Pullback and regular parametricity} \label{ssec:basics_2}
Let $k$ be a field of characteristic zero and $f:X\rightarrow \Pp^1_k$ a $k$-regular cover. Let $T_0\in k(U)\setminus k$; we make no distinction between the rational function $T_0$ and the rational map $T_0:\Pp^1_k \rightarrow \Pp^1_k$ \footnote{In particular, the degree of $T_0\in k(U)$ (the maximum of numerator degree and denominator degree in coprime notation) is the same as the degree of the associated map $\Pp^1_k \rightarrow \Pp^1_k$.}.
The fiber product $X \times_{f,T_0} \Pp^1_k$ provides a cartesian square

\[
\xymatrix{
X \times_{f,T_0} \Pp^1_k \ar[r]^{} \ar[d] & \mathbb P^1_k \ar[d]^{T_0} \\
X \ar[r]^f & \mathbb P^1_k 
}
\]

When $X \times_{f,T_0} \Pp^1_k$ is geometrically irreducible, denote by $f_{T_0}:X_{T_0}\rightarrow \Pp^1_k$ the smooth projective model of the top horizontal map: $X_{T_0}$ is the normalization of $\Pp^1_k$ in the function field of $X \times_{ f,T_0} \Pp^1_k$. The map $f_{T_0}:X_{T_0}\rightarrow \Pp^1_k$ is then a $k$-regular cover, which we call the \textit{\textbf{pullback of $f$ along $T_0$}}. 
More generally, covers $f_{T_0}$ are called \textit{\textbf{rational pullbacks of $f$}}. If $f:X\rightarrow \Pp^1_k$ is additionally assumed to be Galois of group $G$, then $f_{T_0}:X_{T_0}\rightarrow \Pp^1_k$ remains a $k$-regular Galois cover of group $G$.

Given an affine equation $P\in k[T,Y]$ of $f$, consider the subset $H_{P,\overline k(U)}$ of $\overline k(U)$ of all $T_0$ such that $P(T_0(U),Y)$ is irreducible in $\overline k(U)[Y]$. It is a Hilbert subset of $\overline k(U)$ \cite[Chapter 12]{FJ08}. If $T_0 \in H_{P,\overline k(U)}\cap k(U)$, then $T_0\notin k$ and the fiber product $X \times_{f,T_0} \Pp^1_k$ is geometrically irreducible; hence the pullback $f_{T_0}$ is a $k$-regular cover and  $P(T_0(U),Y)$ is a defining equation of $f_{T_0}$. The subset $H_{P,\overline k(U)}\cap k(U)$ only depends on $f$ (and not on the specific equation $P(T,Y)$). Denote it by $H_{f,k}$. The field $\overline k(U)$ being Hilbertian  \cite[Proposition 13.2.1]{FJ08}, the Hilbert subset $H_{P,\overline k(U)}$ is ``big'' in various senses: by definition of a Hilbertian field, it is infinite; by Theorems 3.3 and 3.4 from \cite{Deb99b}, it is dense for the Strong Approximation Topology. The same is true of the set $H_{f,k}$.

\begin{definition} \label{def:main} 
For ${\sf H} \subset {\sf H}_{G}(k)$, we define
$$\hskip 10mm {\rm PB}({\sf H}) =\left\{ f_{T_0} \hskip 2pt \left| \hskip 1mm \begin{matrix}
f \in {\sf H} \hfill \cr 
T_0\in H_{f,k} \hfill \cr
\end{matrix} \right.
\right\} \hskip 5mm \hbox{{\rm (PB for ``PullBack'').}}$$
A subset ${\sf H}$ of ${\sf H}_{G}(k)$ is {\it $k$-\G-parametric} if ${\rm PB}({\sf H}) \supset {\sf H}_G(k)$. We say that a cover $f\in {\sf H}_G(k)$ is {\it $k$-\G-parametric} if the subset  $\{f\}$ of ${\sf H}_G(k)$ is.
\end{definition} 

The $k$-\G-parametricity notion relates to the classical notion of {\it genericity} (in one parameter; see, e.g., \cite{JLY02}): clearly, if a cover $f\in {\sf H}_{G}(k)$ is generic, then it is $k$-\G-parametric. The paper \cite{DKLN1-part-two} says more about how the two notions compare.

\subsection{Prerequisites} \label{ssec:basics_3} 
Let $k$ be an algebraically closed  field of  characteristic $0$.

\subsubsection{Riemann Existence Theorem} \label{sssec:RET}

This fundamental tool of the theory of covers of $\Pp^1$ allows turning questions about covers into combinatorics and group theory considerations. 
\vskip 2mm

\noindent 
{\bf Riemann Existence Theorem (RET).} {\it Given a finite group $G$, an integer $r\geq 2$, a subset ${\bf t}$ of $\Pp^1(k)$ of $r$ points, and an $r$-tuple ${\bf C} = (C_1,\ldots,C_r)$ of non-trivial con\-ju\-ga\-cy classes of $G$, there is a Galois cover $f:X \rightarrow \Pp^1_k$  of group $G$, branch point set ${\bf t}$, and inertia canonical invariant ${\bf C}$ if and only if there exists $(g_1,\ldots,g_r)\in C_1\times \cdots \times C_r$ such that $g_1 \cdots g_r = 1$ and $\langle g_1,\ldots, g_r \rangle=G$. Furthermore, the number of such covers $f:X \rightarrow \Pp^1_k$, counted up to $\Pp^1_k$-isomorphism classes, equals
the number of $r$-tuples $(g_1,\ldots,g_r)$ as above, counted modulo componentwise conjugation by an element of $G$.
}
\vskip 2mm

Cf., e.g., \cite[Theorem 2.13]{Vol96} for the mere existence statement, or \cite{Deb01a} for a detailed overview.

The RET shows that a pair $(r,{\bf C})$ is a ramification type for $G$ over $k$ if the set, traditionally called the \textit{\textbf {Nielsen class}}, of all $(g_1,\ldots,g_r)\in C_1\times \cdots \times C_r$ such that  $g_1 \cdots g_r = 1$ and $\langle g_1,\ldots, g_r \rangle=G$ is non-empty. We shall use the RET to construct Galois covers of given group $G$ and with some special ramification 
type.

\subsubsection{Bounds for the branch point number and the genus of pulled-back covers} \label{sssec:bounds}

\begin{theorem} \label{thm:DeAnnENS}
Let $f:X\rightarrow \Pp^1_k \in {\sf H}_{G}(k)$ and 
$T_0$ be in the Hilbert subset $H_{f,k}$.
Denote the branch point number of $f$ (resp., $f_{T_0}$) by $r$ (resp., $r_{T_0}$) and the genus of $X$ (resp., $X_{T_0}$) by $g$ (resp., $g_{T_0}$). Then $r \leq r_{T_0}$ and $g \leq g_{T_0}$. Moreover, if $g> 1$ and $T_0$ is not an isomorphism, then $g < g_{T_0}$.
\end{theorem}

\begin{proof}
For branch point numbers, see \cite[Theorem 2.1]{Deb18}. Regarding genera, we may assume $g\not=0$ and $T_0$ is not an isomorphism. The claim then follows from applying the Riemann--Hurwitz formula to the cover $X_{T_0}\rightarrow X$. Namely, we obtain $2g_{T_0} - 2 \geq N(2g-2)$ with $N=\deg(f)$, whence $g_{T_0} \geq 2(g-1)+1 \geq g$ if  $g\geq 1$, with equality only if $g=1$.
\end{proof}

We can now explain how Corollary \ref{cor:intro} is deduced from Theorem \ref{thm:main1-plus}.

\begin{proof}[Proof of Corollary \ref{cor:intro} assuming Theorem \ref{thm:main1-plus}] 
First, assume that the group $G$ satisfies the Beckmann--Black regular lifting property over $k$. Pick a Galois cover $g_1\in {\sf H}_{G}(k)$ (such a cover exists from the RET). Let $r_1$ be the branch point number of $g_1$. Then, for any $g_2\in {\sf H}_{G}(k)$, there exists $f\in {\sf H}_{G}(k)$ and $T_{01},T_{02} \in k(U)$ such that $g_i=f_{T_{0i}}$ ($i=1,2$). From Theorem \ref{thm:DeAnnENS}, it follows from $g_1=f_{T_{01}}$ that the branch point number of $f$ is $\leq r_1$. This shows that ${\sf H}_{G,\leq r_1}(k)$ is regularly parametric. From Theorem \ref{thm:main1-plus}, $G\subset {\rm PGL}_2(\Cc)$. 

Conversely, if $G\subset {\rm PGL}_2(\Cc)$, then $G$ has a $k$-regularly parametric cover $f:X\rightarrow \Pp^1_k$ (see \cite[Corollary 2.5]{Deb18}); {\it a fortiori} the Beckmann--Black regular lifting property holds over $k$. 
\end{proof}

\begin{remark} The proofs above of Theorem \ref{thm:DeAnnENS} and Corollary \ref{cor:intro} cite two results which are only stated with $k=\Cc$: Theorem 2.1 and Corollary 2.5 from \cite{Deb18}.
These two results however hold more generally over any algebraically closed field of characteristic $0$. Indeed, regarding \cite[Theorem 2.1]{Deb18}, the assumption $k=\Cc$ was made for simplicity and can readily be generalized. As to \cite[Corollary 2.5]{Deb18}, the main ingredient there is Tsen's theorem that the field $\Cc(U)$ is quasi-algebraically closed, which is also true with $\Cc$ replaced by any algebraically closed field.
\end{remark}

\subsubsection{On varieties and their dimension} \label{sssec:varieties}

In the following, we recall some well-known facts from algebraic geometry about the structure and dimension of images and preimages under algebraic morphisms. 

It is elementary that the image of an $n$-dimensional variety under an algebraic morphism is always of dimension $\le n$. A bound in the opposite direction is given via the dimension of a fiber (see, e.g., \cite[\S 1.8, Theorems 2 and 3]{Mum99}):
 
\begin{theorem} \label{thm:dimbound}
Let $f: X\to Y$ be a dominant morphism between varieties $X$ and $Y$. For any point $p \in f(X)$, we have $\dim(Y)\le \dim(X)\le \dim(Y)+\dim(f^{-1}(p)).$
\end{theorem}

We refer to, e.g., \cite[p. 239, lemme 1.8.4.1]{Gro64} for the next theorem:

\begin{theorem}[Chevalley]\label{thm:chevalley}
Let $f: X\to Y$ be a morphism between varieties $X$ and $Y$. Then the image of any constructible subset of $X$ is constructible\footnote{Here, a subset of a topological space is called {\it{constructible}} if it is a finite union of locally closed sets.}.
\end{theorem}

In particular, the image of any subvariety $X_0$ of $X$ is a finite union of varieties $Y_i \subset Y$, each of dimension at most $\dim(X_0)$.

For short, we shall say that a subset $S$ of a variety $X$ is {\it of dimension $\le d$}, if it is contained in a finite union of subvarieties of dimension $\le d$.

\subsubsection{Defining equations for Galois covers and their pullbacks} \label{sssec:defining}

To prove Theorem \ref{thm:main1-plus}, we shall use affine and defining equations 
of Galois covers of $\mathbb{P}^1_k$ (as defined in \S \ref{sssec:covers}).

\begin{lemma} \label{lem:degree_bound}
Let $G$ be a finite group and $r_0\in \nn$. Then every Galois cover in $\HGlerok$ can be defined by an affine equation $P(T,Y)=0$, where $P \in k[T][Y]$ is irreducible, monic in $Y$, of bounded $T$-degree depending only on $r_0$ and $G$, and of degree $|G|$ in $Y$.
\end{lemma}

\begin{proof}
The fact that every Galois cover of group $G$ and of bounded genus $g\le g_0$ can be defined by an affine equation of bounded $T$-degree depending only on $g_0$ and $|G|$ follows essentially from an application of the Riemann--Roch theorem. Concretely, in \cite[\S2.2]{Sad99} the upper bound $(2g_0+1)|G|\log|G|/\log(2)$ was obtained, cf.\ also \cite[Lemma 4.1]{Deb17}.
It then suffices to note that the genus $g$ of a Galois cover of $\mathbb{P}^1_k$ with $r_0$ branch points is bounded from above only in terms of $r_0$ and $G$: the Riemann--Hurwitz formula gives $2g \le 2-2|G|+r_0(|G|\cdot (1-1/e_{{\rm{max}}}))$, with $e_{{\rm{max}}}$ the maximal element order in $G$.
\end{proof}

\begin{definition} \label{def:pde}
Let $d,e\in \nn$. We denote by $\mathcal{P}_{d,e}$ the space of polynomials $P(T,Y)\in k[T,Y]$ 
of degree exactly $d$ in $T$ and exactly $e$ in $Y$, viewed up to multiplicative constants.
Similarly, let $\mathcal P_{\leq d}$ denote the space of polynomials $Q(T)\in k[T]$ of degree at most $d$, viewed up to multiplicative constants. Furthermore, denote by $\mathcal{R}_d$ the set of rational functions over $k$ in one indeterminate $U$ of degree exactly $d$.
\end{definition}

The spaces $\mathcal{P}_{d,e}$ and $\mathcal{P}_{\leq d}$ are varieties in a natural way, via identifying the polynomials $P(T,Y)=\sum_{i=0}^{d} \sum_{j=0}^{e} \alpha_{i,j}T^iY^j$ and $Q(T)=\sum_{i=1}^d\beta_i T^i$ with the coordinate tuples $(\alpha_{i,j})_{i,j}$ and $(\beta_i)_i$, respectively, in the corresponding projective space. Similarly, $\mathcal{R}_d$ is a variety by identifying a rational function $T_0=T_0(U)=(\sum_{i=0}^d \beta_i U^i)/(\sum_{j=0}^d \gamma_j U^j)$, with coprime numerator and denominator, with the coordinate tuple $(\beta_0: \cdots :\beta_d:\gamma_0: \cdots :\gamma_d) \in \mathbb{P}^{2d+1}$. 
Note that the degree and coprimeness assumptions fix the numerator and denominator up to constant factor, whence the above identification is well-defined.

We can now define pullback maps on the level of the above spaces $\mathcal{P}_{d,e}$ and $\mathcal{R}_d$:

\begin{lemma} \label{pb_morphism}
Let $d_1$, $d_2$, and $d_3$ be positive integers. Then the map
$${\widetilde{{\rm PB}}}: \mathcal{P}_{d_1,d_2}\times \mathcal{R}_{d_3} \to \mathcal{P}_{d_1\cdot d_3, d_2},$$
defined by $(P(T,Y), T_0(U)) \mapsto $ ``numerator of $P(T_0(U),Y)$", is a morphism of algebraic varieties.
\end{lemma}

\begin{proof}
Let $P(T,Y)=\sum_{i=0}^{d_1}\sum_{j=0}^{d_2} \alpha_{i,j}T^iY^j$ and $T_0=(\sum_{k=0}^{d_3}\beta_k U^k) / (\sum_{k=0}^{d_3} \gamma_kU^k)$. 
Then
$${\widetilde{{\rm PB}}}(P,T_0)=\sum_{i,j}\alpha_{i,j} (\sum_{k}\beta_kU^k)^i (\sum_{k}\gamma_kU^k)^{d_1-i} Y^j,$$
and identification with the spaces of coordinate tuples shows that ${\widetilde{{\rm PB}}}$ is given by a polynomial map.
\end{proof}

\section{Proofs of Theorem \ref{thm:main1-plus} and its variants} \label{sec:RP} 

In \S\ref{ssec:intro_1}, we mention two variants of Theorem \ref{thm:main1-plus}: a ``genus $\geq 2$'' version (Theorem \ref{lem:gge2}) and an ``explicit'' variant (Theorem \ref{thm:new}). \S\ref{ssec:Proof-1a} is devoted to the proof of the former, the proof of how Theorem \ref{thm:main1-plus} can be deduced being explained at the end of \S\ref{sssec:reduction} (where Theorem \ref{lem:gge2} is stated). \S\ref{ssec:proof_main1b} is devoted to the proof of Theorem \ref{thm:new}.  
Throughout this section, except in \S \ref{ssec:other_fields}, the field $k$ is algebraically closed of characte\-ristic $0$. 
Our proofs make use of the fact that $k$ is algebraically closed but certain parts carry over to more general fields. We collect such considerations in \S\ref{ssec:genfields}.

\subsection{Proof of Theorem \ref{thm:main1-plus}} \label{ssec:Proof-1a}

\subsubsection{Reduction to Theorem \ref{lem:gge2} and Lemma \ref{lem:geq1}} \label{sssec:reduction}

The following theorem is our strongest result regarding pullbacks of Galois covers of genus at least $2$. Its proof is also the main part of the proof of Theorem \ref{thm:main1-plus}.

\begin{theorem} \label{lem:gge2}
Let $G$ be a finite group, let $r_0\in \nn$, and let 
$${\sf H}_{G,\le r_0, g\ge 2}(k)= \HGlerok\setminus\{\text{genus }\hskip -5pt \le 1 \text{ covers} \}$$ 
be the set of all Galois covers $f:X\rightarrow \Pp^1_k$ of genus at least $2$ with Galois group $G$ and at most $r_0$ branch points. Then there exists $R_0\in \nn$ such that, for every ramification type $(R,{\bf C})$ for $G$ with $R\ge R_0$, we have
$$\HGRCk \not\subset {\rm PB}({\sf H}_{G,\le r_0, g\ge 2}(k)).$$	
In particular, ${\sf H}_{G,\le r_0, g\ge 2}(k)$ is not $k$-\G-parametric.
\end{theorem}

Theorem \ref{lem:gge2} is proved in \S\ref{ssec:Proof-1a}.2-4 below. The following lemma shows non-parametricity for sets of Galois $k$-covers of genus $1$:

\begin{lemma} \label{lem:geq1}
Let ${\sf H}_{G, g=1}(k)$ be the set of Galois covers $f:X\rightarrow \Pp^1_k$ of genus $1$ with group $G$. Then there exists a ramification type $(r,{\bf C})$ for $G$ such that no cover in $\HGrCk$ is a pullback from any cover in ${\sf H}_{G, g=1}(k)$. In particular, ${\sf H}_{G, g=1}(k)$ is not $k$-\G-parametric.
\end{lemma}

\begin{proof}
Let $f:X\to \mathbb{P}^1 \in {\sf H}_{G, g=1}(k)$. As a consequence of the Riemann--Hurwitz formula, the tuple of element orders in the inertia canonical invariant of $f$ is one of $(2,2,2,2)$, $(3,3,3)$, $(2,4,4)$, or $(2,3,6)$. Furthermore, in each case, $G$ has a normal subgroup $N$ with cyclic quotient group $G/N$ of order $2$, $3$, $4$ and $6$, respectively, and such that the quotient map $X\to X/N$ is an unramified cover of genus-$1$ curves over $k$. Assume first $|N|=1$. Then $G$ is cyclic\footnote{In fact, $|G|\in\{2,3,4,6\}$,  since $G$ is then a group of automorphisms of some elliptic curve, see \cite[Chapter III, Theorem 10.1]{Sil09}.} and, therefore, there exist Galois covers of $\mathbb{P}^1_k$ of group $G$ and genus 0. In particular, no set of covers of genus $\ge 1$ can have those as pullbacks, by Theorem \ref{thm:DeAnnENS}.
	
Assume therefore $|N|>1$. Let $x\in N\setminus\{1\}$, and let $(r,{\bf C})$ be any ramification type for $G$ involving the conjugacy class of $x$. Since $X\to X/N$ is unramified, its image under any rational pullback of $f$ must also be unramified. But, of course, for any cover $\widetilde{X}\to \mathbb{P}^1$ with inertia canonical invariant ${\bf C}$, the subcover $\widetilde{X}\to \widetilde{X}/N$ is ramified by definition. Therefore, no cover of inertia canonical invariant ${\bf C}$ can be a pullback of $f$.
\end{proof}

\begin{remark} \label{rem:genus1}
In the case that $G$ is non-cyclic, the above proof shows immediately that, for $(r,{\bf C})$ any ramification type of genus $1$ with group $G$ ($r\in \{3,4\}$) and for each $s\ge r+1$, there exists a ramification type $(s,{\bf D})$ for $G$ such that no $k$-cover in ${\sf H}_{G,s}({\bf D})$ is a pullback of some $k$-cover in $\HGrC$. Indeed, for the only critical case $s=r+1$, it suffices to replace $(x_1,\ldots, x_n)\in {\bf C}$, where $x_1\notin N$ without loss, by $(x_0,x_0^{-1}x_1,\ldots,x_n)$ with $x_0\in N\setminus\{1\}$.
\end{remark}

Assuming Theorem \ref{lem:gge2} and Lemma \ref{lem:geq1}, we can now derive Theorem \ref{thm:main1-plus}. 

\begin{proof}[Proof of Theorem \ref{thm:main1-plus}]
By assumption, there is no Galois cover of $\mathbb{P}^1$ of group $G$ and genus 0. Let $(R,{\bf C})$ be a ramification type for $G$ with ${\bf C} =(C_1,\ldots,C_{R})$. From Theorem \ref{lem:gge2}, we know that not all covers in ${\sf{H}}_{G,R}({\bf{C}})(k)$ are rational pullbacks of some element of ${\sf H}_{G,\le r_0}(k)$ of genus $\ge 2$, if the length $R$ of ${\bf C}$ is sufficiently large (depending on $r_0$). From Lemma \ref{lem:geq1} and its proof, we know that no $k$-cover in ${\sf{H}}_{G,R}({\bf{C}})$ is a rational pullback of some $k$-cover in ${\sf H}_{G,\le r_0}$ of genus $1$, if ${\bf C}$ contains certain conjugacy classes. Altogether, if ${\bf C}$ contains all classes of $G$ sufficiently often, then certainly not all $k$-covers in ${\sf{H}}_{G,R}({\bf{C}})$ are reached via rational pullback of 
some $k$-cover in ${\sf H}_{G,\le r_0}$.
\end{proof}

\subsubsection{Proof of Theorem \ref{lem:gge2}: some dimension estimates} \label{ssec:312}

To prepare the proof of Theorem \ref{lem:gge2}, we investigate the behaviour of rational pullbacks of Galois covers. 

Recall that we have introduced two different ways of associating algebraic varieties to certain sets of Galois covers: the Hurwitz spaces and the spaces of defining equations. In the next lemma, we relate both concepts via a dimension estimate stating, in particular, that, in order to obtain defining equations for all covers in an $r$-dimensional Hurwitz space, we require at least $r$-dimensional subvarieties in the space of defining equations.

To state the lemma, denote by $\mathcal{P}_{d,e}^{\textrm{sep}}$ the subset of separable (in $Y$) polynomials in $\mathcal{P}_{d,e}$ (the latter set is introduced in Definition \ref{def:pde}). Note that this is a dense open subset of $\mathcal{P}_{d,e}$. Due to Lemma \ref{lem:degree_bound}, when looking for defining equations for covers in some $\HGlero(k)$, we can restrict without loss to a suitable finite union of $\mathcal{P}_{d,e}^{\textrm{sep}}$ (for $d$ smaller than some bound depending only on $G$ and $r_0$, and in fact always with $e=|G|$).

\begin{lemma} \label{lem:discriminant}
Let $r,s,d,e\in \nn$. Let $V\subset \mathcal{P}_{d,e}^{\rm{sep}}$ be a subvariety of dimension $s$. Let $G$ be a finite group and let $(r,{\bf C})$ be a ramification type for $G$. Then the set of $f\in \mathcal{H}_{G,r}({\bf C})(k)$ which have a defining equation in $V$ is of dimension $\le s$ \footnote{Note here that equivalent covers have the same defining equations by definition, so that the term ``defining equation for an element of $\mathcal{H}_{G,r}({\bf C})(k)$" is indeed well-defined.}. In particular, if $s<r$, there are infinitely many covers in $\HGrC(k)$ which do not have a defining equation in $V$.
\end{lemma}

\begin{proof}
Denote the discriminant of a polynomial $Q(Y)=\sum_{i=0}^e a_i Y^i$ by 
$$\Delta(Q)= a_e^{2e-2} \prod_{i<j} (r_i-r_j)^2,$$ where the $r_i$'s are the roots of $Q$, counted with multiplicities.
The discriminant induces an algebraic morphism  $\Delta: \mathcal{P}_{d,e}^{\textrm{sep}}\ra \mathcal{P}_{\le c}$, $P(T,Y) \mapsto \Delta(P)$ (where $P$ is viewed as a polynomial in $Y$)  to the space of polynomials in $T$ of degree $\le c$ up to constant factors (see Definition \ref{def:pde}). 
Note that indeed the image of a separable polynomial is nonzero.
The fact that the degree of $\Delta(P)$ is bounded only in terms of $d$ and $e$ follows easily from the fact that the discriminant is a polynomial expression in the coefficients, viewed as transcendentals.
In particular, the dimension of $\Delta(V)\subseteq \mathcal{P}_{\le c}$ is at most $\dim(V) =s$.
	
Next, for any $r\le t\le c$ and any $r$-subset $R$ of $\{1,\ldots,t\}$, consider the morphisms $u: (\mathbb{A}^1)^t \to \mathcal{P}_{\le c}$ given by $(a_1,\ldots,a_t)\mapsto \prod_{i=1}^t (T-a_i)$ and $v: (\mathbb{A}^1)^t \to (\mathbb{A}^1)^r$ the projection on the coordinates in $R$. For each of these finitely many possible maps $u,v$, the map $u$ is finite and so $v(u^{-1}(W))$ is of dimension $\le s$, where $W=\Delta(V)$. But, since any branch point (assumed to be finite without loss) of a Galois cover of $\mathbb{P}^1_k$ is necessarily a root of the discriminant of a defining equation, such a cover can only have a defining equation in $V$ if its branch point set is in $v(u^{-1}(W))$ for some $u,v$ as above. Now, let $\mathcal{U}^r$ and $\mathcal{U}_r$ denote the spaces of ordered and unordered $r$-sets in $\mathbb{P}^1$, respectively. There is a well-defined finite morphism from $\mathcal{H}_{G,r}(\bf C)$ to $\mathcal{U}_r$: the branch point reference map. Now, let ${\mathcal{H}}'= \mathcal{H}_{G,r}({\bf {C}}) \times_{\mathcal{U}_r} \mathcal{U}^r$ be the ``ordered branch point set version" of the Hurwitz space $\mathcal{H}_{G,r}(\bf C)$. Then the branch point reference map induces a finite morphism $\psi: \mathcal{H}' \to \mathcal{U}^r (\subset (\mathbb{P}^1)^r)$. Furthermore, there is a natural finite morphism $\pi:\mathcal{H}'\ra \mathcal{H}_{G,r}(\bf C)$. 
	
In particular, each set $\pi(\psi^{-1}(v(u^{-1}(W))))$, and thus finally also the set of $f\in \mathcal{H}_{G,r}({\bf C})(k)$ having (only finite branch points and) a defining equation in $V$, is of dimension $\le s$. 

The additional assertion in the case $s<r$ follows immediately, since $\psi({\mathcal{H}}')\cap (\mathbb{A}^1)^r$ is of dimension $r$ and, in fact, equal to the set of all ordered $r$-sets in $\mathbb{A}^1$, by the Riemann Existence Theorem. 
\end{proof}

In the proof of Theorem \ref{lem:gge2}, we shall show as an intermediate result that a rational function pulling a prescribed Galois cover $f$ of $\mathbb{P}^1_k$ back into a prescribed $\HGRC(k)$ can only have a certain maximal number of branch points outside of the branch point set of $f$.
We then require the following auxiliary result stating that varieties of rational functions with such partially prescribed branch point sets cannot be too large.

\begin{lemma} \label{lem:dimension_ratfct}
Let $d,m,n,s\in \mathbb{N}$, let $W$ be a subvariety of $\mathcal{P}_{m,n}^{\textrm{sep}}$, and let $\mathcal{R}_d$ be as in Definition \ref{def:pde}. Then the set $W'\subset W \times \mathcal{R}_d$ of all $(P,T_0) \in W \times \mathcal{R}_d$ such that at most $s$ branch points of $T_0$ are not roots of $\Delta(P)$ is of dimension at most $\dim W + s +3$.
\end{lemma}

\begin{proof}
Denote the discriminant map on $\mathcal{P}_{m,n}^{\textrm{sep}}$ by $\Delta_1$ and the one on $\mathcal{R}_d$ by $\Delta_2$. Here we define the discriminant of a rational function $T_0(U)=T_{0,1}(U)/T_{0,2}(U)$ as the discriminant of the polynomial $T_{0,1}(U)-T\cdot T_{0,2}(U)$ with respect to $U$. Note that, in the special case of rational functions, every root of the discriminant is in fact a branch point (see, e.g., \cite[Lemma 3.1]{Mue02} for a stronger version of this statement). This means that, with $u:(a_1,\ldots,a_t) \mapsto \prod_{i=1}^t (T-a_i)$ as before, an element of $u^{-1}(\Delta_2(T_0))$ is already the exact branch point set of $T_0$, up to multiplicities.
	
Consider now the following chain of maps:
$$W' \subset W \times \mathcal{R}_d \hskip 2mm \rhfl{{\rm id}\times \Delta_2}{}   \hskip 2mm W \times \mathcal{P}_{\le t} \hskip 2mm \lhfl{{\rm id}\times u}{}  \hskip 2mm W\times (\mathbb{A}^1)^t \hskip 2mm \lhfl{\alpha}{}  \hskip 2mm W\times (\mathbb{A}^1)^t,$$
where $\alpha$ is defined by $\alpha(P,(a_1,\ldots,a_t))= (P, (\Delta_1(P)(a_1),\ldots,\Delta_1(P)(a_t)))$. Clearly, all maps in this chain are morphisms and, except for the first map ${\rm id}\times \Delta_2$, they are all finite. Therefore, $({\rm id}\times \Delta_2)(W')$ is of the same dimension as $\alpha^{-1}\bigl(({\rm id} \times (u^{-1}\circ \Delta_2))(W')\bigr)$ 
and, by definition of $W'$, the latter is contained in one of finitely many varieties isomorphic to $W \times (\mathbb{A}^1)^s$. Indeed, up to repetitions, all except for at most $s$ roots of $\Delta_2(T_0)$ are mapped to $0$ under $\Delta_1(P)$, for $(P,T_0)\in W'$. Thus, $({\rm id}\times \Delta_2)(W')$ is of dimension at most $\dim(W)+s$. Theorem \ref{thm:dimbound} then yields that $W'$ is of dimension at most $\dim(W) + s +\dim(\Delta_2^{-1}(p))$ for any $p$ equal to the discriminant of a rational function $T_0$ as above.
	
It remains to show that $\Delta_2^{-1}(p)$ is of dimension $\le 3$. Now, the set of genus zero covers $\mathbb{P}^1_k\to \mathbb{P}^1_k$ (viewed up to equivalence) of degree $d$ with prescribed branch point set is finite, and each such cover is given by a degree-$d$ rational function, unique up to ${\rm PGL}_2(k)$-equivalence. Since $\dim({\rm PGL}_2)=3$, the claim follows, completing the proof.
\end{proof}

\subsubsection{Proof of Theorem \ref{lem:gge2}: reduction to Lemma \ref{lem:tbd}}  \label{ssec:lemma}

\begin{lemma}\label{lem:tbd}
Let $G$ be a finite group and $f:X\to\mathbb{P}^1_k$ a Galois cover with group $G$ and genus $\ge 1$. Then, for every $j\in \nn$, there exists a constant $R_0\in \nn$, depending only on $j$ and the branch point number of $f$, such that, for every class-$R$-tuple ${\bf C}$ of $G$ ($R\ge R_0$) and for every rational function $T_0:\mathbb{P}^1_k\to \mathbb{P}^1_k$ in $H_{f,k}$ with more than $R-j$ branch points outside the branch point set of $f$, the pullback of $f$ along $T_0$ is not in $\HGRC(k)$.
\end{lemma}

\begin{proof}[Proof of Theorem \ref{lem:gge2} assuming Lemma \ref{lem:tbd}]
Let $f: X\to \mathbb{P}^1 \in {\sf H}_{G,\le r_0, g\ge 2}(k)$. Let $g$ be the genus of $X$. Let $F=F(T,Y)$ be a separable defining equation for $f$, of minimal degree in $T$. Using the pullback map ${\widetilde{{\rm PB}}}$ as in Lemma \ref{pb_morphism}, denote by ${\widetilde{{\rm PB}}}(f)$ the set of all pullbacks of $F$ by rational functions of arbitrary degree, i.e., ${\widetilde{{\rm PB}}}(f)=\cup_{d\in \nn} {\widetilde{{\rm PB}}}(\{F\} \times \mathcal{R}_d)$. Then ${\widetilde{{\rm PB}}}(f)$ contains defining equations for all rational pullbacks of the cover $f$.
	
Let $(R,{\bf C})$ be a ramification type for $G$. By the Riemann--Hurwitz formula, the genus of a Galois $k$-cover of group $G$ arising as a degree-$d$ pullback of $f$ is $\geq d(g-1)+1$. As $g\ge 2$, this shows that there is $d_0\in \nn$, depending only on ${\bf C}$, such that, for all $d>d_0$, a degree-$d$ pullback of $f$ cannot have inertia canonical invariant ${\bf C}$ (as the genus is the same for all covers with invariant ${\bf C}$). In other words, to investigate the set of polynomials in ${\widetilde{{\rm PB}}}(f)$ which are defining equations for covers in $\HGRC(k)$, it suffices to restrict to pullback functions $T_0\in \mathcal{R}_{d}$, $d\le d_0$, with some bound $d_0\in \nn$ depending only on ${\bf C}$.
	
Let $D\in \nn$ be such that every $f\in \HGlero(k)$ has a separable defining equation in some space $\mathcal{P}_{d_1,|G|}$ with $d_1\le D$. Such $D$ exists by Lemma \ref{lem:degree_bound}. Let $\delta$ be the dimension of $\mathcal{P}_{D,|G|}$ (to be explicit, $\delta=(D+1)(|G|+1)-1$). 
	
Fix $j>\delta+3$, choose $R_0$ large enough (see the proof of Lemma \ref{lem:tbd} for an explicit bound) and $R\ge R_0$, and denote by $\mathcal{S}_f$ the set of   $T_0 \in H_{f,k}$ for some  $f\in {\sf H}_{G,\le r_0, g\ge 2}(k)$ such that $f_{T_0}$ has $R$ branch points.  
As seen above, the degree of such $T_0$ is absolutely bounded from above (in terms of the genus, and thus the branch point number of $f$) and, by Lemma \ref{lem:tbd}, all $T_0\in \mathcal{S}_f$ have at most $R-j$ branch points outside the branch point set of $f$. {\it A fortiori}, they are in the set $\mathcal{S}'_F$ of rational functions (of bounded degree as before and) with at most $R-j$ finite branch points outside the set of roots of the discriminant of $F(T,Y)$, for a defining equation $F(T,Y)=0$. The latter sets $\mathcal{S}'_F$ can be defined for all $F\in \mathcal{P}_{d_1,d_2}^{\rm{sep}}$ (not just for those defining Galois covers).

Now, consider the set $\mathcal{S}=\cup_{d_1\le D}\cup \{F\}\times \mathcal{S}'_F$, where the inner union is over all $F\in \mathcal{P}_{d_1,|G|}^{\rm{sep}}$. From Lemma \ref{lem:dimension_ratfct} (with $W=\mathcal{P}_{d_1,|G|}^{\rm{sep}})$, it follows that $\mathcal{S}$ is contained in a a finite union of varieties, of dimension at most $\dim(\mathcal{P}_{D,|G|}) + R-j+3 = \delta+R-j+3 < R$.
	
Therefore, the image of $\mathcal{S}$ under ${\widetilde{{\rm PB}}}$ is of dimension strictly smaller than $R$ as well. On the other hand, Lemma \ref{lem:discriminant} shows that no finite union of varieties of dimension $<R$ can contain defining equations for all $k$-covers in $\HGRC$.
	
Hence, ${\sf H}_{G,\le r_0, g\ge 2}(k)$ is not $k$-\G-parametric. 
\end{proof}

\begin{remark} \label{rem:gen1_overC}
In fact, the restriction to sets of Galois covers of genus $\ge 2$ is used only once in the proof of Theorem \ref{lem:gge2}; namely, to ensure that the degree of a rational function $T_0$, that pulls back a $k$-cover with $\le r_0$ branch points to a $k$-cover in  $\HGRC$, is bounded from above in terms of $r_0$ and $R$. This then ensures, via the various auxiliary lemmas, that ${\rm PB}({\sf H}_{G,\le r_0, g\ge 2}(k)) \cap \HGRC(k)$ is contained in a finite union of lower-dimensional varieties (as soon as $R$ is sufficiently large), i.e., that its complement inside $\HGRC(k)$ contains a Zariski-dense open subset. If the set ${\sf H}_{G,\le r_0, g\ge 2}(k)$ in Theorem \ref{lem:gge2} is replaced by ${\sf H}_{G,\le r_0, g\ge 1}(k)$, this strong conclusion will no longer be guaranteed. However, since all the auxiliary lemmas remain valid, we obtain in the same way that ${\rm PB}({\sf H}_{G,\le r_0, g\ge 1}(k)) \cap \HGRC(k)$ is contained in a {\it countable} union of lower-dimensional varieties (each corresponding to rational functions $T_0$ of some fixed degree). This at least implies $\HGRC(k) \not\subset {\rm PB}({\sf H}_{G,\le r_0, g\ge 1}(k))$ as soon as $k$ is uncountable. Thus the conclusion of Theorem \ref{lem:gge2} remains valid for ${\sf H}_{G,\le r_0, g\ge 1}(k)$ in the important special case $k=\mathbb{C}$.
\end{remark}

\subsubsection{Proof of Lemma \ref{lem:tbd}} \label{ssec:314} 

Let $t_1,\ldots,t_s$ be the branch points of $f$, and let $e_1,\ldots,e_s$ be the corresponding orders of inertia groups of $f$. Set $e_{{\rm{max}}}=\max\{e_1,\dots ,e_s\}$. We choose $R_0 = (s-2+j)e_{{\rm{max}}}$, and $R\ge R_0$ arbitrary. We divide the proof into two main steps.

\vskip 2mm

\noindent
{\it First step: Translation into a combinatorial statement.} Let $\sigma_1,\ldots,\sigma_s\in S_d$ be the inertia group generators of $T_0$ at $t_1,\ldots,t_s$, and $\sigma_{s+1},\ldots,\sigma_{s+m}$ the non-trivial inertia group gene\-rators at further points. For $\sigma\in S_d$, denote by $o(\sigma)$ the number of orbits of $\langle\sigma\rangle$, and set ${\rm{ind}}(\sigma)=d-o(\sigma)$. We claim that, assuming choice of $R_0$ as above, the following holds:

\vspace{2mm}

\noindent
{\bf Claim:}
$$\sum_{i=1}^s {\rm{ind}}(\sigma_i) \ge 2d-2-R+j,$$ or equivalently:
$$
\sum_{i=1}^s o(\sigma_i) \le d(s-2)-j+2 + R.
$$
The assertion then follows from the claim, 
since $T_0$ defines a genus-zero cover, whence the Riemann--Hurwitz formula yields $\sum_{i=1}^{s+m} {\rm{ind}}(\sigma_i) =2d-2$. Together with the claim, this enforces $m\le R-j$.

\vskip 2mm

\noindent
{\it Second step: Transformation of cycle structures.} To prove the claim, consider the cycle structures of $\sigma_i$, $i=1,\ldots,s$. We shall manipulate the cycle structures of the $\sigma_i$ in a controlled way, to make it easier to estimate the total number of orbits of all $\sigma_i$, $i=1,\ldots,s$.

By the definition of $T_0$ and Abhyankar's lemma, the cycle lengths of the $\sigma_i$ are multiples of $e_i$, for $i=1,\ldots,s$, with a total of exactly $R$ exceptions over all $i\in \{1,\ldots,s\}$. For each $i\in \{1,\dots,s\}$, let $n_{i,1}$,\dots, $n_{i,r(i)}$ denote the exceptional cycle lengths (i.e., the ones which are not multiples of $e_i$), in descending order. Define a permutation $\tau_i\in S_d$ in the following way. First, fill all the non-exceptional cycles of $\sigma_i$ into $\tau_i$. Next, find the smallest $j$ such that $\sum_{k=1}^j n_{i,k} \ge e_i$ and, instead of the cycles of length $n_{i,1}, \dots, n_{i,j}$, fill one single cycle of length $\sum_{k=1}^j n_{i,j}$ into $\tau_i$. Repeat this procedure until the sum of the remaining exceptional cycle lengths in $\sigma_i$ is less than $e_i$. Fill one more cycle of length the sum of those remaining exceptional cycle lengths into $\tau_i$. By definition, all except possibly one cycle of $\tau_i$ have length $\ge e_i$. In particular, the number $o(\tau_i)$ of orbits of $\langle \tau_i\rangle$ is bounded from above by $\lceil d/e_i \rceil < d/e_i + 1$. Also, since $f$ is of genus $\ge 1$, the Riemann--Hurwitz formula yields $\sum_{i=1}^s \frac{1}{e_i} \le s-2$.

Therefore, 
$$\sum_{i=1}^s o(\tau_i) < d(\sum_{i=1}^s \frac{1}{e_i}) + s \le d(s-2) + s.$$
On the other hand, we can effectively bound the difference between $\sum_{i=1}^s o(\tau_i)$ and $\sum_{i=1}^s o(\sigma_i)$ via the above construction. Namely, to obtain $(\tau_1,\dots,\tau_s)$ from $(\sigma_1,\dots,\sigma_s)$, certain sets of exceptional cycles were replaced by one big cycle. There were $R$ exceptional cycles in total and, since at most $e_{{\rm{max}}}$ cycles each were replaced by one cycle, the total difference $\sum_{i=1}^s \bigl(o(\sigma_i) - o(\tau_i)\bigr)$ is bounded from above by $R- \frac{R}{e_{{\rm{max}}}}$. Altogether, 
$$\sum_{i=1}^s o(\sigma_i) \le R - \frac{R}{e_{{\rm{max}}}} + \sum_{i=1}^s o(\tau_i) < R - \frac{R}{e_{{\rm{max}}}} + d(s-2)+s \le R - (s-2+j) + d(s-2)+s,$$
with the last inequality following from our choice of $R_0$. This finally yields
$$\sum_{i=1}^s o(\sigma_i) \le d(s-2)-j+2 + R,$$ 
showing the claim. This completes the proof.

\subsection{An explicit variant of Theorem \ref{thm:main1-plus}}  \label{ssec:proof_main1b} 

The goal of this subsection is to prove the following statement, alluded to in \S\ref{ssec:intro_1} as an explicit variant of Theorem \ref{thm:main1-plus}.

\begin{theorem} \label{thm:new}
Let $k$ be an algebraically closed field of characteristic $0$ and $G$ a finite group with at least $5$ maximal non-conjugate cyclic subgroups. Fix an integer $r_0\geq 0$. For every suitably large even in\-te\-ger $R$, there is a non-empty Hurwitz stack ${\sf H}_{G,R}({\bf C})$ such that {\textbf{no}} $k$-cover in ${\sf H}_{G,R}({\bf C})$ is a rational pullback of some $k$-cover in ${\sf H}_{G,\leq r_0}$.
\end{theorem}

\begin{remark} 
Subgroups of ${\rm PGL}_2(\Cc)$ have at most $3$ maximal non-conjugate cyclic subgroups. Other groups have exactly $3$: the quaternion group $\Hh_8$ or, more generally, dicyclic groups ${\rm DC}_n$ of order $4n$ ($n\geq 2$) (which include generalized quaternion groups ${\rm DC}_{2^{k-1}}$ $k\geq 2$), groups ${\rm SL}_2(\Ff_q)$ with $q$ a prime power, {etc}. For these groups, the conclusion from Theorem \ref{thm:main1-plus} holds but that from Theorem \ref{thm:new} is unclear. Replacing $5$ by $4$ in Theorem \ref{thm:new} seems feasible but hard and technical; for the sake of brevity, we avoid this slight improvement. Groups 
with $4$ maximal non-conjugate cyclic subgroups (for which the conclusion from Theorem \ref{thm:new} might also hold) include  $\Zz/4\Zz\times \Zz/2\Zz$, $\Zz/3\Zz\times \Zz/3\Zz$, $A_6$.
\end{remark}

\begin{proof} 
Fix a finite group $G$ with at least $5$ non-conjugate maximal cyclic subgroups. At first, assume that not all of them are of order $2$. Let $\gamma_1,\ldots, \gamma_5$ be generators of $5$ non-conjugate maximal cyclic subgroups of $G$, and let ${\mathcal C}_1,\ldots, {\mathcal C}_5$ be their conjugacy classes. Denote the order of $\gamma_i$ by $e_i$, $i=1,\ldots,5$ and, without loss of generality, assume $e_1>2$. Consider then a tuple $({\mathcal C}_1,\ldots,{\mathcal C}_5, {\mathcal C}_6, \ldots, {\mathcal C}_s)$
of conjugacy classes of $G$, not necessarily distinct, and satisfying the following:
\begin{align*}\label{equ:repeat}
\tag{A} & \text{ all the non-trivial conjugacy classes of $G$, but the powers  ${\mathcal C}_i^j$, $i=1,\ldots, 4$, } \\
& \text{ $j=1,\ldots, e_i-1$, appear in the set $\{{\mathcal C}_5,\ldots, {\mathcal C}_s$\}. }
\end{align*}
Consider the $(2s)$-tuple $\underline{\mathcal C} = (\mathcal C_1,\mathcal C_1^{-1},\ldots, \mathcal C_s, \mathcal C_s^{-1})$.  Note that the integer $s$ can be taken to be any suitably large integer, e.g., by repeating the conjugacy class ${\mathcal C}_5$. 

Picking $g_i\in \mathcal C_i$ ($i=1, \dots, s$), form the tuple $\underline{g}=(g_1,g_1^{-1},\ldots,g_s,g_s^{-1}$). As the elements of $\underline g$ and their powers contain at least one element from each conjugacy class, a classical lemma of Jordan (see \cite{Jor72}) implies that $\underline{g}$ forms a generating set of $G$. By construction, the product of entries of $\underline{g}$ is $1$. From the Riemann Existence Theorem, ${\sf H}_{G,2s}(\underline{\mathcal C})(k) \not= \emptyset$.

Let $h \in {\sf H}_{G,2s}(\underline{\mathcal C})(k)$. Assume there exist $r_0\in \Nn$, a Galois cover $f\in {\sf H}_{G,\leq r_0}(k)$, and $T_0\in H_{f,k}$ (as defined in \S  \ref{ssec:basics_2}) of degree $N$ such that $h$ and the pullback $f_{T_0}$ are $\Pp^1_k$-isomorphic. Denote the branch point number of $f$ by $r$ (so $r\leq r_0$) and its inertia canonical invariant by ${\bf C}=(C_{f,1}, \ldots, C_{f,r}$). By \cite[\S 3.1]{Deb18}, the inertia canonical invariant of $f_{T_0}$ is a tuple ${\bf C}_{f,T_0}$ obtained by concatenating tuples of the form ${\bf C}_{f,T_0,j}=(C_{f,j}^{e_{j1}}, \ldots, C_{f,j}^{e_{jr_j}}),\, j=1,\ldots, r,$ where $r_j,e_{j,\ell}$ are integers with $r_j\geq 0$ and $e_{j,\ell}\geq 1$ for all $\ell=1,\ldots,r_j$, and $j=1,\ldots,r$. Note that some of the classes in ${\bf C}_{f,T_0}$ might be trivial. 

Denote by $p_j$ (resp., $q_j$) the number of $e_{j,\ell}$'s, $\ell=1,\ldots,r_j,$ that are equal to $1$ (resp., $>1$), for $j=1,\ldots,r$. For $j=1$, denote further by $u_1$ (resp., $v_1$) the number of $e_{1,\ell}$'s, $\ell=1,\ldots,r_1,$ that are equal to $2$ (resp., $>2$). Recall further from \cite[\S 3.1]{Deb18} that, since $e_{j1},\ldots,e_{jr_j}$ are the ramification indices of $T_0$ over some point, we have
$\sum_{\ell=1}^{r_j}e_{j,\ell} = N$, for $j=1,\ldots,r$. Hence, $p_j + 2q_j\leq N$, and $p_1+2u_1 + 3v_1\leq N$, or equivalently: 
\begin{equation}\label{equ:bounds-pqst}
q_j\leq \frac{N-p_j}{2}\leq \frac{N}{2} \text{ for }j=2,\ldots,r\text{ and }v_1\leq \frac{N-p_1-2u_1}{3}\leq \frac{N}{3}. 
\end{equation}
By the Riemann--Hurwitz formula for $T_0$, we also have 
\begin{equation}\label{equ:RH-T0-Pierre}
2N-2\geq \sum_{\ell=1}^{r_1}(e_{1,\ell}-1) +\sum_{j=2}^r \sum_{\ell=1}^{r_j}(e_{j,\ell}-1) = N-p_1-u_1-v_1 + \sum_{j=2}^r(N-p_j-q_j). 
\end{equation}

As $f_{T_0}$ and $h$ are $\Pp^1_k$-isomorphic, we have ${\bf C}_{f,T_0} = \underline{\mathcal C}$, up to order. Without loss of generality, we may assume $\mathcal C_1\in {\bf C}_{f,T_0,1}$. For each $i=1,\ldots, 4$, if $j_i\in \{1,\ldots,r\}$ is such that ${\mathcal C}_i\in {\bf C}_{f,T_0,j_i}$ \footnote{There may be {\it a priori} two different $j_i\in \{1,\ldots,r\}$ such that ${\mathcal C}_i\in {\bf C}_{f,T_0,j_i}$.}, then, as $\langle \gamma_i \rangle$ is a maximal cyclic subgroup of $G$, we have ${\mathcal C}_i =  {C}_{f,j_i}^{w_i}$ for some integer $w_i$ relatively prime to $e_i$. This, together with the assumption that $\mathcal C_1,\ldots,\mathcal C_4$ are not powers of each other, implies that the correspondence $i\mapsto j_i$ is injective. Thus, \eqref{equ:RH-T0-Pierre} and \eqref{equ:bounds-pqst} give
$$2N-2 \geq   N-p_1-u_1-v_1 + \sum_{i=2}^4(N-p_{j_i}-q_{j_i}) \geq   \frac{13}{6}N  - (p_1+u_1) - \sum_{i=2}^4p_{j_i},  $$
and hence  
\begin{equation}\label{equ:bound-N-p_1} 
N \leq 6(p_1+u_1+\sum_{i=2}^4p_{j_i}).
\end{equation}
 Since each $\mathcal C_i$, $i=1,\ldots,4$, appears at most  twice in ${\bf C}_{f,T_0}$ by (A), it follows that $p_{j_i}\leq 2$ for $i=1,\ldots,4$. Furthermore, since the powers of $\mathcal C_1$ appear at most twice and $e_1>2$, we also have $p_{1}+u_{1}\leq 2$. 
Thus, \eqref{equ:bound-N-p_1} gives $N\leq 48$. However, by a priori choosing $s$ to be large, the degree $N$ of $T_0$ is forced to be at least $49$, contradiction. 

It remains to consider the case where all the fixed $5$ non-conjugate maximal cyclic subgroups are of order $2$. We can reduce to the previous case by changing one of them to another maximal cyclic subgroup of order $>2$, unless all elements of $G$ are of order $2$. But, in this case, $G$ is an elementary abelian $2$-group of rank at least $3$, and the number of maximal conjugacy classes of cyclic groups is at least $7$.  Repeating the above argument with seven classes $\mathcal C_1,\ldots,\mathcal C_7$, replacing the previous $\mathcal C_1,\ldots,\mathcal C_5$, \eqref{equ:RH-T0-Pierre} and \eqref{equ:bounds-pqst} take the form
$$2N-2 \geq   \sum_{j=1}^r(N-p_j-q_j)\text{ and } q_{j_i} \leq \frac{N-p_{j_i}}{2} \leq \frac{N}{2} \text{ for }i=1,\ldots,6.$$
Thus, their combination gives 
$$ 2N-2  \geq    \sum_{i=1}^6(N-p_{j_i}-q_{j_i}) \geq  3N -\sum_{i=1}^6p_{j_i}, $$
and so $N \leq \sum_{i=1}^6p_{j_i}$. Once again, choosing $s$ and hence $N$ sufficiently large, we obtain a contradiction. 
\end{proof}

\subsection{Extension to more general fields} \label{ssec:other_fields} \label{ssec:genfields} 

We have assumed so far that the field $k$  is algebraically closed because we use the Riemann Existence Theorem. However, as we explain below, this assumption can be relaxed in some situations.

\begin{theorem} \label{thm:main-general-fields} 
Let $G$ be a finite group and let $k$ be a field of characteristic $0$.

\vskip 0,25mm

\noindent
{\rm (a)} Theorem \ref{thm:main1}{\rm{(b)}}  with $\Pp^1_{\Cc}$ replaced by $\Pp^1_k$ holds in each of these situations:

\vskip 0,25mm

\hskip 3mm {\rm (a-1)} $k$ is ample\footnote{Definition of ``ample field'' is recalled in Remark \ref{rk:intro}(a).},

\vskip 0,25mm

\hskip 3 mm  {\rm (a-2)} $G$ is abelian of even order,

\vskip 0,25mm

\hskip 3 mm  {\rm (a-3)} $G$ is the direct product $A\times H$ of an abelian group $A$ of even order and of a non-

\hskip 12mm  solvable group $H$ occuring as a regular Galois group over $k$ \footnote{In fact, the assumption on $H$ to be non-solvable can be removed with a bit of extra effort. Moreover, if $H$ is not a regular Galois group over $k$, then $G$ is not either. Hence, Theorem \ref{thm:main1}{\rm{(b)}} trivially fails.},

\vskip 0,25mm

\hskip 3 mm  {\rm (a-4)} $G=S_5$.

\vskip 0.25mm

\noindent
{\rm (b)} Theorem \ref{thm:new} holds if either $k$ is ample or $G$ is abelian.
\end{theorem}

\begin{remark} The proof below actually gives a little more than what is
stated. Under the assumptions of Theorem \ref{thm:main-general-fields}(a), not only can we conclude that ${\sf H}_G(k) \not\subset {\rm{PB}}({\sf H}_{G,\le r_0}(k))$ for any integer $r_0\geq 0$ (as is stated), but what we really obtain is that 
\vskip 1,5mm

\centerline{${\sf H}_G(k)\otimes_k \overline k \hskip 1,5mm \not\subset \hskip 0,5mm {\rm{PB}}({\sf H}_{G,\le r_0}(\overline k))$.}
\vskip 1,5mm

\noindent
That is: for any $r_0\geq 0$, one can always find a $k$-cover in ${\sf H}_{G}$ that cannot be obtained, even after base change to $\overline k$, by pulling back some $\overline k$-cover in ${\sf H}_{G, \leq r_0}$ (and even 
along some non-constant rational map $T_0 \in \overline k(U)$ (and not just in $k(U)$). Similarly, if $k$ is ample or $G$ is abelian as in Theorem \ref{thm:main-general-fields}(b), what the proof shows is that for any $r_0\geq 0$ and any suitably large $s$, there is a Hurwitz stack ${\sf H}_{G,2s}(\underline{\mathcal C})$ such that ${\sf H}_{G,2s}(\underline{\mathcal C})(k)\not=\emptyset$ and 
\vskip 1,5mm

\centerline{${\rm{PB}}({\sf H}_{G,\le r_0}(\overline k)) \hskip 1mm \cap \hskip 1mm ({\sf H}_{G,2s}(\underline{\mathcal C})(k) \otimes_k \overline k) = \emptyset$,}
\vskip 1,5mm

\noindent
and not just ${\rm{PB}}({\sf H}_{G,\le r_0}(k)) \cap {\sf H}_{G,2s}(\underline{\mathcal C})(k) = \emptyset$.

\noindent
\end{remark}

\subsubsection{Proof of Theorem \ref{thm:main-general-fields}(b)}

In the proof of Theorem \ref{thm:new} (see \S\ref{ssec:proof_main1b}), the assumption $k=\overline k$ was only used to guarantee that there is at least one $k$-cover in the Hurwitz stack ${\sf H}_{G,2s}(\underline{\mathcal C})$. When $k$ is no longer algebraically closed, we first slightly modify the tuple $\underline{\mathcal C}$ to make it {\it $k$-rational}, {i.e.}, such that the action of $\Gal(\overline k/k)$ on $\underline{\mathcal C}$ (taking the power of the classes by the cyclotomic character) preserves $\underline{\mathcal C}$, up to the order. This can be done by replacing, for each index $i=1,\ldots,4$, 

\vskip 1mm
 
\centerline{each pair $({\mathcal C}_i,{\mathcal C}_i^{-1})$ by the tuple $({\mathcal C}_i,{\mathcal C}_i^{-1},,\ldots, {\mathcal C}_i^{e_i-1},{\mathcal C}_i^{-(e_i-1)}$).}

\vskip 1mm

\noindent
The modified tuple is indeed $k$-rational and the proof of Theorem \ref{thm:new} still holds after slight adjustments: the inequalities $p_{j_i} \leq 2$ (resp., $p_1+u_1\leq 2$) become $p_{j_i} \leq 2(e_i-1)$ (resp., $p_1+u_1\leq 2(e_1-1)$), $i=1,\ldots,4$,  changing only the constants in the final estimates.

With the tuple $\underline{\mathcal C}$ now $k$-rational, it is still true that the Hurwitz stack ${\sf H}_{G,2s}(\underline{\mathcal C})$ contains a $k$-cover, so that Theorem \ref{thm:new} holds, in the following situations:

\noindent
- $G$ abelian and $k$ arbitrary (as a consequence of the classical rigidity theory; see, e.g., \cite[Chapter I, \S4]{MM18} and \cite[\S3.2]{Vol96}),

\noindent
- $k$ ample and $G$ arbitrary. Namely, recall that, over the complete discretely valued field $k((T))$, the so-called $1/2$-Riemann Existence Theorem of Pop (see \cite{Pop94}) ensures that ${\sf H}_{G,2s}(\underline{\mathcal C})(k((T)))$ is non-empty. This implies ${\sf H}_{G,2s}(\underline{\mathcal C})(k)\neq\emptyset$ for an ample field $k$, as then $k$ is existentially closed in $k((T))$ (see \cite[Proposition 1.1]{Pop96} and \cite[\S4.2]{DD97b}).

\subsubsection{Proof of Theorem \ref{thm:main-general-fields}{\rm (a)}} \label{ssec:PAC}

In the following, we will denote equivalence classes $[f]$ in $\mathcal{H}_{G,r}({\bf C})$ of a cover $f$ 
(of group $G$ and ramification type $(r,\bf C)$)
simply by $f$ as well when there is no risk of confusion.

We shall deduce Theorem \ref{thm:main-general-fields}(a) from Theorem \ref{thm:main_arb_fields} below. Start with an arbitrary characteristic $0$ field $k$ and consider the following condition, for a finite group $G$:

\begin{condition} \label{cond:manycovers}
{\rm{(a)}} There exist a constant $m\ge 0$, infinitely many integers $R$ and, for each $R$, a ramification type $(R,{\bf C})$ for $G$ 
with the following property: the set of all equivalence classes $[f]\in \mathcal{H}_{G,R}(\bf C)$ such that $f$ is a $k$-regular Galois cover cannot be covered by finitely many varieties of dimension $<R-m$ {\rm (}in $\mathcal{H}_{G,R}({\bf C})(\overline{k})${\rm )}.

\vspace{0.5mm}

\noindent
{\rm{(b)}} Each ${\bf C}$ in {\rm (a)} contains every conjugacy class of $G$ at least once.
\end{condition}

We then have the following analog of Theorem \ref{lem:gge2} and Lemma \ref{lem:geq1}:

\begin{theorem} \label{thm:main_arb_fields}
Assume $G$ fulfills Condition \ref{cond:manycovers}(a) (resp., Condition \ref{cond:manycovers}(a) and (b)) over a field $k$ of characteristic $0$. Let $r_0\in \nn$, and let ${\sf H}={\sf H}_{G,\le r_0, g\ge 2}(k)$ (resp., ${\sf H}={\sf H}_{G,\le r_0, g\ge 1}(k)$) be the set of $k$-regular Galois covers with group $G$, branch point number $\le r_0$, and genus $\ge 2$ (resp., genus $\ge 1$). Then there are infinitely ramification types $(R,\bf C)$ for $G$  such that $\HGRCk \not\subset {\rm PB}({\sf H})$. In particular, ${\sf H}$ is not $k$-regularly parametric.
\end{theorem}

\begin{proof}
Observe that the crucial Lemma \ref{lem:tbd} in the proof of Theorem \ref{lem:gge2} (see \S \ref{ssec:lemma}) guarantees this: there exists $R_0\in \nn$ such that, for every ramification type $(R,{\bf C})$ for $G$ ($R\ge R_0$), the set of (defining equations of) Galois covers in ${\sf{H}}_{G,R}({\bf{C}})(\overline{k})$ which arise as rational pullbacks of some cover with $\le r_0$ branch points is contained in a union of finitely many varieties of dimension at most $R-(m+1)$. Then, with $R$ sufficiently large and $(R,{\bf C})$ as in Condition \ref{cond:manycovers}, it follows as in Lemma \ref{lem:discriminant} that these varieties are not sufficient to yield defining equations for every cover in ${\sf H}_{G,R}({\bf C})(\overline{k})$ which is defined over $k$. Furthermore, the additional Condition \ref{cond:manycovers}(b) is sufficient for the proof of Lemma \ref{lem:geq1} over $k$, if $G$ is non-cyclic. For cyclic $G=\zz/n\zz$, Lemma \ref{lem:geq1} holds as soon as the genus-$0$ extension $k(\sqrt[n]{T})/k(T)$ is $k$-regular, i.e., as soon as $e^{2i\pi/n}\in k$. On the other hand, Galois groups of genus-$1$ cyclic covers are only $\zz/n\zz$ for $n\in \{2,3,4,6\}$, and it is easy to verify that, for $e^{2i\pi/n}\notin k$, these genus-$1$ covers cannot be defined over $k$ either, as a special case of the {\it{Branch Cycle Lemma}} (see \cite{Fri77} and \cite[Lemma 2.8]{Vol96}).
\end{proof}

\begin{proof}[Proof of Theorem  \ref{thm:main-general-fields}{\rm (a-1)}] Let $k$ be an ample field of characteristic zero. Equivalently to the definition of ``ample'', every 
 $k$-variety with a simple $k$-rational point has a Zariski-dense set of $k$-rational points (see \cite[Lemma 5.3.1]{Jar11}). On the other hand, the $1/2$-Riemann Existence Theorem yields plenty of class $r$-tuples ${\bf C}$ of $G$ with $\HGrCk \not= \emptyset$. For example, all ${\bf C}$ corresponding to an arbitrarily long repetition of the tuple $(x_1,x_1^{-1},\ldots,x_n,x_n^{-1})$, where $x_i$ runs through all non-identity elements of $G$, are fine. For $Z(G)= \{1\}$, this then implies the existence of a $k$-rational point on $\mathcal{H}_{G,r}(\bf C)$ (cf.\ \S \ref{sssec:Hurwitz}) and  as $k$ is ample that  at least one connected component of $\mathcal{H}_{G,r}(\bf C)$ has a Zariski-dense set of $k$-rational points. Since these $k$-rational points are 
 equivalence classes of covers $f$ defined over $k$ and all components are of the same dimension, Condition \ref{cond:manycovers} holds, even with $m=0$, and, therefore, Theorem \ref{thm:main_arb_fields} holds over $k$. 
 Finally, the case of arbitrary $G$ can be reduced to the above assumption $Z(G)=\{1\}$ by elementary means, using that every finite group  is a quotient of a group with trivial center. Concretely, for a group $G$ of order $n$, choose any non-abelian simple group $S$ and consider the wreath product $\Gamma:= S\wr G =S^n \rtimes G$. 
 Then $\Gamma$ clearly has trivial center. Furthermore, the conjugates of a given complement $G\le \Gamma$ of $S^n$ generate all of $\Gamma$. Therefore, if $(r,{\bf C})$ is a  ramification type of $G$ with $\mathcal{H}_{G,r}({\bf C})(\overline{k})\ne \emptyset$, then for a suitable multiple $R$ of $r$, there exists a ramification type $(R, {\bf \tilde{C}})$ of $\Gamma$ projecting onto a repetition of ${\bf C}$, and such that $\mathcal{H}_{\Gamma, R}(\tilde{\bf C})(k)\ne \emptyset$. As above, the equivalence classes of $k$-regular covers of group $\Gamma$ and ramification type (any arbitrarily long repetition of) ${\bf \tilde{C}}$ are Zariski-dense in at least one connected component, thus fulfilling Condition \ref{cond:manycovers}. But by the choice of ${\bf \tilde{C}}$, for any such cover, the subcover with group $G$ is a $k$-regular cover with the same branch point set, thus yielding Condition \ref{cond:manycovers} also for $G$.
\end{proof}

\begin{proof}[Proof of Theorem \ref{thm:main-general-fields}{\rm (a-2)}]

As a first example, let $G$ be an elementary abelian $2$-group. Since $G$ is abelian, every tuple $(C_1,\ldots,C_R)$ of conjugacy classes in $G$ with non-empty Nielsen class is a {\it rigid} tuple (in the sense of, e.g., \cite[Chapter I, \S4]{MM18}).

Furthermore, since all non-identity elements of $G$ are of order $2$, every conjugacy class is trivially {\it rational} (i.e., unchanged if taken to a power relatively prime to the order of its elements). It then follows from the rigidity method that, for every choice $(t_1,\ldots,t_R)$ of $R$ distinct points in $\mathbb{P}^1(k)$, there exists a Galois cover of $\mathbb{P}^1$, defined over $k$, with inertia canonical invariant $(C_1,\ldots,C_R)$ and branch points $(t_1,\ldots,t_R)$. In particular, the set of all these covers cannot be obtained by a set of defining equations of dimension $<R$. Hence, $G$ fulfills Condition \ref{cond:manycovers} over $k$. The assertion of Theorem \ref{thm:main_arb_fields} therefore holds for $G$ over an arbitrary field of characteristic zero.

Now, let $G$ be an arbitrary abelian group of even order. As before, all class tuples with non-empty Nielsen class are rigid. Let $(C_1,\ldots,C_R)$ be a class tuple of $G$ containing each element of order $\ge 3$ exactly once, and each element of order $2$ an arbitrary even number of times. This then yields a product-$1$ tuple generating $G$ and, if the branch points for each set of generators of a cyclic subgroup $\zz/d\zz$ are chosen appropriately to form a full set of conjugates (under the action of ${\textrm{Gal}}(\qq(e^{2i\pi/d})/\qq$)) in $k(e^{2i\pi/d})$, then the associated ramification type is a rational ramification type, implying that there is again a $k$-regular Galois cover with inertia canonical invariant $(C_1,\ldots,C_R)$ and with the prescribed branch point set. Note that the branch points for elements of order $2$ are still allowed to be chosen freely in $k$. Since there are less than $|G|$ other branch points, it follows as above that the set of Galois covers with these ramification data cannot be obtained by a set of defining equations of dimension $\le R-|G|$. Therefore again, Condition \ref{cond:manycovers} is fulfilled. 
\end{proof}

\begin{proof}[Proof of Theorem \ref{thm:main-general-fields}{\rm (a-3)}] 
Let $G=A\times H$, where $A$ is an abelian group of even order and $H$ is any non-solvable group which occurs as a regular Galois group over $k$. We use the non-solvability assumption only to obtain that there are no Galois covers of genus $\le 1$ with group $G$. Take a tuple $(C_1,\ldots,C_R)$ of classes of $A\le G$ as in the proof of Theorem \ref{thm:main-general-fields}{\rm (a-2)}, and prolong it by a fixed tuple $(C_{R+1},\ldots,C_S)$ of classes which occurs as some ramification type for $H$ over $k$. With the appropriate choice of branch point set for the $H$-cover, we obtain a $k$-regular Galois cover with group $A\times H$ where, once again, the branch points with involution inertia in $A$ can be chosen freely (outside of the fixed branch points of the $H$-cover). Increasing $R$ as in the proof of Theorem \ref{thm:main-general-fields}{\rm (a-2)} (whilst fixing $(C_{R+1},\ldots,C_S)$), we again obtain that Condition \ref{cond:manycovers}(a) is fulfilled, and so Theorem \ref{thm:main_arb_fields} holds over $k$ for Galois covers of genus $\ge 2$. As there is no Galois cover of group $G$ and genus $\le 1$, Theorem \ref{thm:main1}(b) holds for $G$ over all fields of characteristic $0$.
\end{proof}

\begin{proof}[Proof of Theorem \ref{thm:main-general-fields}{\rm (a-4)}] 
Let $M_{g,d}$ be the moduli space of simply branched covers (i.e., all non-trivial inertia groups are generated by transpositions) of degree $d$ and genus $g$. It is known (see \cite{AC81}) that $M_{g,5}$ is unirational for all $g\ge 6$ and, in fact, this holds even over the smallest field of definition $\mathbb{Q}$ (and, hence, over all fields of characteristic zero).

Note that $M_{g,5}$ parameterizes Galois covers in ${\sf{H}}_{S_5, 8+2g}({\bf{C}})$ with ${\bf{C}}=(C_1,\ldots,C_{8+2g})$, where each $C_i$ is the class of transpositions. This space is of dimension $8+2g$, and unirationality (over $k$) implies that its function field is of finite index in some $k(T_1,\ldots,T_{8+2g})$ with independent transcendentals $T_i$. But, of course, every $k$-rational value of $(T_1,\ldots,T_{8+2g})$ then leads to a $k$-rational point on $M_{g,5}$ (and, thus, a cover defined over $k$), and the set of such $k$-rational points on a unirational variety is always Zariski-dense.

This implies that Condition \ref{cond:manycovers}(a) is fulfilled and, as there is no Galois cover of group $S_5$ and genus $\le 1$, Theorem \ref{thm:main1}(b) holds for $S_5$ over all fields of characteristic $0$.
\end{proof}

\section{Proof of Theorem \ref{introthm:r+1}} \label{sec:RP2} 

Throughout this section, fix an algebraically closed field $k$ of characteristic $0$ and a Galois cover $f:X\ra \mP^1_k$ with group $G$ and ramification type $(R,\bf C)$ for $R\geq 4$.  Assume $G \not \subset \PGL_2(\mathbb C)$. Let ${\bf t}=\{t_1,\ldots,t_R\}$ be the branch point set of $f$. Let $e_f(t_0)$ denote the ramification index of $t_0\in\mP^1(k)$ under $f$, and set $e_i=e_f(t_i)$ for $i=1,\ldots,R$. Given $T_0\in k(U)\setminus k$ and $t_0\in \mathbb P^1(k)$, let $e(q|t_0)$ denote the ramification index under $T_0$ of $q\in T_0^{-1}(t_0)$. The proof is based on the following estimate on the branch point number of a pullback, which strengthens the bounds in \cite[Theorem 3.1(b-2)]{Deb18}. Recall that for  $T_0\in H_{f,k}$, the fiber product $X\times_{f,T_0}\mP^1_k$ is irreducible, cf.\ Section \ref{ssec:basics_2}.

\begin{lemma}\label{lem:rT0} 
Let $T_0\in H_{f,k}$ be of degree $n$. 
Let $a_i$ be the number of preimages $q\in T_0^{-1}(t_i)$ with $e_{i}\divides e(q|t_i)$ for $i=1,\ldots,R$, and $U_{T_0,f}$ the set of points $q\in\mP^1(k)$ such that $e(q|t_0)\neq e_f(t_0)$ for $t_0=T_0(q)$. Then the branch point number $R_{T_0}$ of  the pullback $f_{T_0}$ is at least 
$$R_{T_0}\geq  (R-4)n + 4 + \sum_{i=1}^R(e_{i}-2)a_i + \sum_{q\in U_{T_0,f}}\bigl(e(q|t_0)-1\bigr).$$
Moreover, equality holds if and only if $T_0$ is unramified away from ${\bf t}$ and its ramification indices over $t_i$ are either $e_{i}$ or not divisible by  $e_{i}$, for $i=1,\ldots,R$. 
\end{lemma}

\begin{proof} 
Let $b_i$ denote the number of preimages $q$ in $T_0^{-1}(t_i)$ such that  $e(q|t_i)$ is not divisible by $e_{i}$, for $i=1,\ldots,R$. 
Note that, since $X\times_{f,T_0}\mP^1_k$ is irreducible, Abhyankar's lemma implies
\begin{equation}\label{equ:rT0}
R_{T_0} = b_1 + \cdots + b_R.
\end{equation}
By the Riemann--Hurwitz formula for $T_0$, we have
\begin{equation}\label{equ:RH} 2n-2 = \sum_{t_0\in\mP^1}\sum_{q\in T_0^{-1}(t_0)}(e(q|t_0)-1) \geq \sum_{q\in U_{T_0,f}}\bigl(e(q|t_0)-1\bigr)+  \sum_{i=1}^R (e_{i}-1)a_i,
\end{equation}
with equality if and only if $e(q|t_i)$ is either $e_{i}$ or non-divisible by $e_{i}$, for all points $q\in T_0^{-1}(t_i), i=1,\ldots,R$. The same Riemann--Hurwitz formula also implies $$2n-2\geq \sum_{i=1}^R(n-a_i-b_i) = Rn -\sum_{i=1}^R(a_i+b_i),$$
with equality if and only if $T_0$ is unramified away from $\bf t$. Combined with \eqref{equ:RH}, this gives
\begin{align*} 2n -2 & \geq Rn - \sum_{i=1}^R(a_i+b_i) \\
& \geq Rn + \sum_{i=1}^R(e_{i}-2)a_i + \sum_{q\in U_{T_0,f}}\bigl(e(q|t_0)-1\bigr) - (2n-2) - \sum_{i=1}^Rb_i.
\end{align*}
Combined with \eqref{equ:rT0}, this gives
$$R_{T_0} = \sum_{i=1}^Rb_i \geq (R-4)n + 4 + \sum_{i=1}^R(e_{i}-2)a_i + \sum_{q\in U_{T_0,f}}\bigl(e(q|t_0)-1\bigr),$$
with equality if and only if $T_0$ is unramified away from ${\bf t}$ and every ramification index $e(q|t_i)$, for $q\in T_0^{-1}(t_i)$, $i=1,\ldots,R$, which is divisible by $e_i$ is equal to $e_i$. 
\end{proof} 

Let $E$ denote the multiset $\{e_1,\ldots,e_R\}$. 
Throughout the proof below, we assume $T_0$ is in the Hilbert subset $H_{f,k}$, and hence that the fiber product $X\times_{f,T_0}\mP^1_k$ is irreducible.

\begin{proof}[Proof of Theorem \ref{introthm:r+1} when $E\neq \{2,2,2,3\}, \{2,2,2,4\}$]
The case where $f$ is of genus $1$ follows from Remark \ref{rem:genus1}, so henceforth we shall assume that the genus of $X$ is at least $2$. Let $(g_1, \ldots,g_R)$ be a tuple in the Nielsen class of ${\bf C}$, corresponding to $f$. As any permutation of the tuple $\bC=(C_1,\ldots,C_R)$ has a non-empty Nielsen class, without loss of generality, we may assume that $g_i$ is a branch cycle over $t_i$ and the orders $e_1,\ldots,e_R$ of $g_1,\ldots,g_R$ are ordered in decreasing order. This tuple can be modified to a tuple ($P_y$) $y, y^{-1}g_1,\ldots,g_R$ for any $y\neq g_1$. Such a tuple generates $G$ and has product $1$, giving a non-empty Nielsen class corresponding to a ramification type which we denote by $(R+1,{\bf D}_y)$. We will show that, for a suitable choice of $y\in G$, no  pullback $f_{T_0}$, along $T_0\in k(U)\setminus k$ of degree $n>1$,  is in ${\sf H}_{G,R+1}({\bf D}_y)(k)$. Since a cover  in ${\sf H}_{G,R+1}({\bf D}_y)(k)$ is a pullback of $f$ only if $y$ is a power of some element in $C_1,\ldots,C_R$, by varying $y$, we may  assume that every conjugacy class in $G$ is a power of one of $C_1,\ldots,C_R$. 

As in Lemma \ref{lem:rT0}, let $a_i$ (resp., $b_i$) be the number of preimages $q$ in $T_0^{-1}(t_i)$ such that the ramification index $e(q|t_i)$ is divisible by $e_{i}$ (resp., is not divisible by $e_i$) for $i=1,\ldots,R$. For $R\geq 6$, as $n\geq 2$, the lower bound on $R_{T_0}$ from Lemma \ref{lem:rT0} is $ \geq R+2$, as desired\footnote{We note that, for $R\geq 6$, this claim also follows for any choice of $y$ from \cite[Theorem 3.1(b-2)]{Deb18}, since the latter implies that every pullback of a $k$-cover in ${\sf{H}}_{G,R}(\bf C)$ has at least $R+2$ branch points, and hence is not in $\sH_{G,R+1}({\bf D}_y)(k).$}.

{\bf The case $R=5$:} We may assume $R_{T_0}=6$. By Lemma \ref{lem:rT0}, we have
\begin{equation*}\label{equ:r=5} 
n\leq 2 - \sum_{i=1}^R(e_i-2)a_i -\sum_{q\in U_{T_0,f}}\bigl(e(q|t_0)-1\bigr).
\end{equation*} 
As $n>1$, this forces $n=2$, $e(q|t_0)=1$ for all $q \in U_{T_0,f}$, and $\sum_{i=1}^R(e_i-2)a_i=0$. Thus, every ramification index of $T_0$ over $t_i$ is either $e_i$ or $1$, and $T_0$ is unramified away from ${\bf t}$. Since $T_0$ is of degree $2$, it ramifies over exactly two branch points with ramification index $2$. Hence, without loss of generality, we may assume $e_4=e_5=2$, so that the inertia canonical invariant of $f_{T_0}$ is ${\bf C}_{T_0}=(C_1,C_1,C_2,C_2,C_3,C_3)$. If $e_1>2$, by picking $y$ to be an involution, no cover 
in ${\sf H}_{G,R+1}({\bf D}_y)(k)$ is a pullback of $f$, as ${\bf D}_y$ has more conjugacy classes of involutions than ${\bf C}_{T_0}$ does. Similarly, if $e_1=2$ and $G$ contains an element $y$ of order $>2$, then every cover in ${\sf H}_{G,R+1}({\bf D}_y)(k)$  is not a pullback of $f$, as its Nielsen class has a conjugacy class of non-involutions. If $e_1=2$ and all elements of $G$ are of order $2$, then $G$ is an elementary abelian $2$-group generated by three conjugacy classes $C_1,C_2,C_3$. As $G \not \subset \PGL_2(\mathbb C)$, this forces $G\cong (\mathbb Z/2\mZ)^3$. In the latter case, we may choose $y$ to be an involution whose conjugacy class is different from $C_1,C_2,C_3, C_4,C_5$, so that no cover in ${\sf H}_{G,R+1}({\bf D}_y)(k)$ is a pullback of $f$. 

{\bf The case $R=4$:} Assume $R_{T_0}=5$. By Lemma \ref{lem:rT0}, we have
\begin{equation}\label{equ:r=4}
\sum_{i=1}^R(e_i-2)a_i + \sum_{q\in U_{T_0,f}}\bigl(e(q|t_0)-1\bigr) \leq 1.
\end{equation}
Let ${\bf t}_2$ (resp., ${\bf t}_{>2}$) denote the set of $t_i$ with $e_{i}=2$ (resp., $e_{i}>2$). Since every point $q\in T_0^{-1}(t_i)$, $t_i\in {\bf t}_2$, with ramification index $e(q|t_0)>2$ contributes at least $2$ to \eqref{equ:r=4}, we deduce that 
$e(q|t_i)=1$ or $2$ for all $q\in T_0^{-1}(t_i)$, $t_i\in {\bf t}_2$. On the other hand, for $t_i\in {\bf t}_{>2}$, \eqref{equ:r=4} gives $a_i=0$ with the possible exception of a single $t_{\iota}$ for which $a_{t_{\iota}}=1$ and $e_{\iota}=3$. 
Moreover, if such $\iota$ exists, then the contribution of the sum over $U_{T_0,f}$ in \eqref{equ:r=4} is $0$. Hence,
\begin{equation}\label{equ:cond-iota}
\begin{array}{l}\text{$e(q|t_i)=1$ for all points $q\in T_0^{-1}(t_i)$, $t_i\in {\bf t}_{>2}$, with the possible } \\
\text{exception of a single $q_\iota\in T_0^{-1}(t_\iota)$ where $e_\iota=e(q_\iota|t_\iota)=3$.}
\end{array}
\end{equation}
Otherwise, $a_i=0$ for all $t_i\in {\bf t}_{>2}$, in which case \eqref{equ:r=4} shows that  
\begin{equation}\label{equ:cond-eta}
\begin{array}{l} 
\text{$e(q|t_i)=1$ for all points $q\in T_0^{-1}(t_i)$, $t_i\in {\bf t}_{>2}$, with the possible } \\
\text{exception of a single $q_\eta\in T_0^{-1}(t_\eta)$ where $e(q_\eta|t_\eta)=2\neq e_\eta$.}
\end{array}
\end{equation} 

If $G$ is of odd order, \eqref{equ:cond-iota} and \eqref{equ:cond-eta} force $T_0$ to have at most one ramification point, which, by the Riemann--Hurwitz formula, forces a contradiction to $n>1$. Henceforth, assume $G$ is of even order. 

Let $a=|{\bf t}_{>2}|$. Since the genus of $X$ is more than $1$ and $R=4$,  the Riemann--Hurwitz formula for $f$ implies $a\geq 1$. Assume first $a>1$. In this case, we let $y\in G$ be an involution. The number of elements of order $>2$ in $(P_y)$ is $a$ or $a-1$. 
In case there exists $\iota$ as above, \eqref{equ:cond-iota} forces every point $t_i\in {\bf t}_{>2}\setminus \{t_\iota\}$ to have at least three unramified preimages and hence forces $f_{T_0}$ to have at least $3(a-1)$ branch points with ramification index $>2$. As $3(a-1)>a$ for $a>1$, the pullback $f_{T_0}$ does not not have ramification type $(R+1,{\bf D}_y)$. The same argument applies if  \eqref{equ:cond-eta} holds and $n\geq 3$. If $n=2$ and \eqref{equ:cond-eta} holds, then $f_{T_0}$ has at least $2(a-1)+1$ branch points with ramification index $>2$. Similarly, $2(a-1)+1>a$, and hence $f_{T_0}$ does not have ramification $(R+1,{\bf D}_y)$. 
 
Henceforth, assume $a=1$, that is, $E=\{e,2,2,2\}$. Note that $e> 4$ by assumption. Condition \eqref{equ:cond-iota} does not hold since $e>3$, and hence \eqref{equ:cond-eta} does. The latter implies that the inertia canonical conjugacy classes of $f_{T_0}$ over points in $T_0^{-1}(t_1)$ are either $C_1$ or $C_1^2$, hence of order $e$ or $e/2$. Thus, if we pick $y$ to be of order different from $e$ and $e/2$, then the only element of $(P_y)$ that can appear in such conjugacy class is $yx_1$. Thus, \eqref{equ:cond-eta} implies: 
\begin{equation}\label{equ:cond-eta2}
\begin{array}{l} 
\text{$T_0^{-1}(t_1)$ consists of a single point $q_\eta$ with ramification $e(q_\eta|t_1)=2$ such that} \\ 
\text{the inertia canonical invariant of $f_{T_0}$ over $q_\eta$ is the conjugacy class of $yx_1$.}
\end{array}
\end{equation}
The latter then has to be of order $\widetilde e=e/2$ or $e$. 

At first, consider the case where the conjugacy classes $C_2,C_3,C_4$ do not coincide. Without loss of generality, we may then assume that $C_2$ is different from $C_3$ and $C_4$. Picking $y=x_3$, \eqref{equ:cond-eta2} implies that  $n=2$, and $(R+1,{\bf D}_y)$ is a ramification type consisting of conjugacy classes of orders $2,\widetilde e,2,2,2$, with $\widetilde e=e$ or $e/2$. In particular, $\widetilde e>2$ and $C_2$ appears at most once in ${\bf D}_y$. Since $n=2$ and $f_{T_0}$ is assumed to have $5$ branch points, $T_0$ has to ramify over exactly one of the places $t_2,t_3,t_4$, say $t_k$. If $k=2$, then $C_2$ does not appear in the ramification type of $f_{T_0}$, contradicting its appearance in ${\bf D}_y$. If $k=3$ or $4$, then $t_2$ is unramified under $T_0$ and hence $C_2$ appears at least twice in the ramification type of $f_{T_0}$, but only once in ${\bf D}_y$, contradiction.  

Now, consider the case $C_2=C_3=C_4$. Note that, since $G$ is non-cyclic, $C_2$ is not a power of $C_1$. Assume next that $e$ is even. We may then pick $y$ to be an involution which is a power of $x_1$. As $y$ is not of order $e$ or $e/2$, we get $y\not\in C_1\cup C_1^2\cup C_2$, contradicting that those are the only conjugacy classes that may appear in the ramification type of $f_{T_0}$. Next, assume $e$ is odd, and pick a prime $p$ dividing $e$. If $p=e$, then $G$ is solvable by Burnside's theorem. Letting $N$ be a minimal normal subgroup of $G$, it follows that $N$ contains exactly one of $C_1$ and $C_2$. The product $1$ relation then gives a contradiction in $G/N$. Thus, we may assume $p$ is a proper divisor of $e$. In this case, if we pick $y$ to be a power of $x_1$ of order $p$, then \eqref{equ:cond-eta2} implies that the only conjugacy classes appearing in the ramification type of $f_{T_0}$ are $C_1,C_1^2$ or $C_2$, neither of which contains $y$, contradiction. This concludes the proof in the case $E\neq \{2,2,2,3\}, \{2,2,2,4\}$.
\end{proof}

\begin{remark}\label{rem:magma} 
In the following, we shall use Magma for computations with small order groups. More specifically, we use the command ExtensionsOfElementaryAbelianGroup to run over extensions of a given group by an elementary abelian group.
\end{remark}
 
Finally, we show that, if $E$ is $\{3,2,2,2\}$ or $\{4,2,2,2\}$, then either $G\subseteq \PGL_2(\mathbb C)$ or there is no group $G$ whose maximal conjugacy classes are the classes of a product $1$ tuple $x_1,\ldots,x_4$ with orders in $E$. Since, in both cases, $|G|$ is divisible by at most two primes, $G$ is solvable by Burnside's theorem. Let $N$ be a minimal normal subgroup of $G$, so that $N\cong (\mZ/2\mZ)^u$ or $(\mZ/3\mZ)^u$, for $u\geq 1$. 

\vspace{2mm}

\noindent {\bf Case \{3,2,2,2\}:} If $N\cong (\mZ/3\mZ)^u$, the images of $x_2,x_3,x_4$ in $G/N$ remain of order $2$ and hence $G/N\cong (\mZ/2\mZ)^2$. Since, moreover, $G/N$ acts transitively on the $\mZ/3\mZ$ subgroups of $N$, it follows that $u=1$ or $2$.  A check using Magma (see Remark \ref{rem:magma}) shows that all such group extensions $G$ contain an element of order~$6$. 

In the case $N\cong (\mZ/2\mZ)^u$, the subgroup $N$ contains exactly one of the conjugacy classes $C_2,C_3,C_4$ since, otherwise, we get a contradiction to the product $1$ relation in $G/N$. Since $G/N$ is generated by three elements of orders $3,2$ and $2$ with product $1$, it is isomorphic to $S_3$. As $G/N$ acts transitively on the $\mZ/2\mZ$ copies in $N$, we have $u=1$ or $2$. Once again, a Magma check shows that all such group extensions $G$ contain an element of order $4$ or $6$, contradicting that $C_i$, $i=1,\ldots,4$ are the only maximal ones.

\vspace{2mm}

\noindent {\bf Case \{4,2,2,2\}:} In this case, $G$ is a $2$-group, and $N\cong (\mZ/2\mZ)^u$. The conjugacy classes of involutions in $G$ are $C_1^2,C_2,C_3$, and $C_4$. If $N$ does not contain $C_1^2$, then the product one relation in $G/N$ implies that $N$ contains exactly one of the conjugacy classes $C_2,C_3,C_4$. In this case, $G/N$ is dihedral of order $8$, acting transitively on the $\mZ/2\mZ$ copies in $N$. Hence, $u=1$ or $2$. A Magma check (see Remark \ref{rem:magma}) shows that such a $G$ either contains an element of order $8$, or  has more than four conjugacy classes of involutions, or has more than two conjugacy classes of elements of order $4$ (in which case there is more than one conjugacy class of cyclic subgroups of order $4$). 

If $N$ contains $C_1^2$, then all elements of $G/N$ are involutions, and hence $G/N$ is an elementary abelian $2$-group. Since we may assume $N$ is a proper subgroup of $G$, the product $1$ relation in $G/N$ implies that $N$ contains one or two of the conjugacy classes $C_2,C_3,C_4$. In the former case, $G/N$ is a $2$-group acting transitively on involutions in $N$, forcing $u=1$. In this case, a Magma check shows that, for such group extensions, either the number of conjugacy classes of involutions is more than $4$ or the number of conjugacy classes of order $4$ elements is more than $2$. If $N$ contains two of $C_2,C_3,C_4$, then $G/N\cong \mZ/2\mZ$. Thus, by minimality of $N$, we get that $u\leq 2$ and hence that $G$ is isomorphic to a subgroup of $\PGL_2(\mathbb C)$.

\bibliography{Debes_et_al}
\bibliographystyle{alpha}

\end{document}